\RequirePackage{fix-cm}
\documentclass[12pt, letterpaper]{article}

\usepackage{amsmath, amssymb, amsthm}
\usepackage{graphicx} 
\usepackage{nickmath}
\usepackage{xcolor}
\usepackage[round]{natbib}
\usepackage{setspace}
\usepackage{tikz}
\usetikzlibrary{calc,arrows.meta,decorations.pathreplacing}
\usepackage{placeins}
\usepackage{enumitem}
\usepackage{eurosym}

\newtheorem{theorem}{Theorem}
\newtheorem{lemma}{Lemma}
\newtheorem{proposition}{Proposition}
\newtheorem{corollary}{Corollary}
\theoremstyle{definition}
\newtheorem{definition}{Definition}

\newtheorem{compositeassumption}{Assumption}

\newtheorem{testingaxiom}{Axiom}
\renewcommand{\thetestingaxiom}{T\arabic{testingaxiom}}
\newtheorem{evalueaxiom}{Axiom}
\renewcommand{\theevalueaxiom}{E\arabic{evalueaxiom}}
\newtheorem{extendedcalibrationaxiom}{Axiom}

\newtheorem{numericalaxiom}{Axiom}

\newtheorem{bettingaxiom}{Axiom}

\newtheorem{birnbaumaxiom}{Axiom}

\theoremstyle{remark}
\newtheorem{remark}{Remark}
\newtheorem{example}{Example}

\title{\bf Axioms for validity: from tests, e-values and p-values to  decisions}
\author{Nick W. Koning\footnote{n.w.koning@ese.eur.nl, Econometric Institute, Erasmus University Rotterdam.}}
\date{23 November 2026}

\begin{document}
	\maketitle
	\begin{abstract}
		The emerging literature on hypothesis testing with data-dependent significance levels relies on a particular extension of the Type-I error.
		Existing arguments for this extension are heuristic.
		We show it is uniquely characterized by three axioms: law-invariance, agreement with classical testing on constant levels, and the preservation of (in)validity under independent selection of the level.
		We find that three analogous axioms characterize the e-value as a multi-decision generalization of a test.
		We also characterize the p-value, reducing the p-vs-e debate to a simple property: should independent selection among evidence variables that are invalid under the same null distribution be able to produce a valid evidence variable?
		Our main technical contribution is to show how e-values emerge in every decision problem with a notion of validity that satisfies law invariance, the independent selection axiom, and an axiom stating that a procedure that selects an invalid procedure with arbitrarily high probability must itself be invalid.
		We also extend recent results using e-values to support decision making to arbitrary evidence variables.
		In an application, we justify expectation-based error rates in multiple testing and recover the `e-closure principle'.
	\end{abstract}

	\section{Introduction}
		The hypothesis testing framework of Neyman \& Pearson has developed into one of the methodological pillars of modern empirical research.
		The central idea in this framework is to choose a test that is \emph{valid} at some significance level $\alpha$:
		\begin{align*}
			\textnormal{Type-I error} \equiv P(\textnormal{test falsely rejects the hypothesis at level } \alpha) \leq \alpha.
		\end{align*}

		In this classical framework, it is mandatory to choose the significance level in advance, or at least independently from the data.
		The standard argument is that for a data-driven choice of the level $\widetilde{\alpha}$, there is generally no hope that the (conditional) Type-I error given the chosen level $\widetilde{\alpha} = a$ is always at most $a$
		\begin{align}\label{ineq:distortion}
			P(\textnormal{test falsely rejects the hypothesis at level } \widetilde{\alpha} \mid \widetilde{\alpha} = a) \leq a.
		\end{align}

		This classical Neyman \& Pearson framework was recently shaken by a series of papers on testing with data-dependent and even fully post-hoc selected significance levels.
		The key idea in this literature is to settle for a weaker target: controlling the expected distortion ratio between the conditional Type-I error and the reported level
		\begin{align}\label{ineq:expected_distortion_ratio}
			\Ex^P\left[\frac{P(\textnormal{test falsely rejects hypothesis at level }\widetilde{\alpha} \mid \widetilde{\alpha})}{\widetilde{\alpha}}\right] \leq 1.
		\end{align}

		Unfortunately, the arguments for this extension of the Type-I error to data-dependent levels are heuristic, and perhaps primarily driven by a resulting connection to the e-value.

		This leaves open the question:
		\begin{quote}
			\emph{Is this the `right' extension of the classical Type-I error to data-dependent levels?}
		\end{quote}
		In particular: why should \eqref{ineq:expected_distortion_ratio} involve a ratio?
		And why should it be bounded in expectation?
		We believe that answering this question is critical in order to convince anyone to go beyond a framework as well established as that of Neyman \& Pearson.

		\subsection{Axioms for evidence: tests, e-values and p-values}		
			Our first contribution is to show that \eqref{ineq:expected_distortion_ratio} is the only notion of validity that satisfies three axioms.
			The first axiom is that validity should only depend on the distribution of the reported decision: the frequency at which each decision is made.
			The second axiom is that it must agree with classical validity when the level is fixed, so that it is a proper extension of the Type-I error.
			This axiom is responsible for the ratio in \eqref{ineq:expected_distortion_ratio}.
			The third axiom is responsible for the expectation operator.
			It states that if a test is valid at certain fixed levels, then it must remain valid when one of these levels is selected independently of the data.
			Similarly, if the test is invalid at each of these levels, the resulting procedure must be invalid.
			
			Our second contribution is to characterize the e-value using three parallel axioms.
			For this purpose, we consider \emph{evidence variables} as multi-decision generalizations of a level $\alpha$ test.
			Beyond the binary `not reject' and `reject at level $\alpha$' decisions of a level $\alpha$ test, such an evidence variable returns a decision in $\{\textnormal{not reject}, \textnormal{reject at level } \alpha_1, \textnormal{reject at level } \alpha_2, \dots\}$.
			The existing motivation that connects such evidence variables to e-values is heuristic \citep{koning2024continuous}.
			We show that law invariance, calibration to classical testing and preservation of (in)validity under independent selection among evidence variables together uniquely characterize the e-value as a numerical representation of an evidence variable, with $e = 1/\alpha$ corresponding to `reject at level $\alpha$' and $e = 0$ to `not reject'.
			
			
			Our third contribution is to connect our axioms to the p-value, which we view as another numerical representation of an evidence variable.
			We find that replacing independent selection with monotonicity and comonotone aggregation axioms characterizes p-value validity.
			
			These axiomatizations enable a concrete comparison between e-values and p-values.
			We show that their difference can be reduced to a single property: should a procedure that almost surely reports the outcome of an invalid procedure be possibly valid?
			In particular, every evidence variable that is p-valid but e-invalid has a distribution that is an independent mixture of invalid evidence variables.
			We illustrate this in Example \ref{exm:invalid_mixture_p_values}.

			\begin{example}[Selecting between invalid p-values]\label{exm:invalid_mixture_p_values}
				Suppose we have two candidate p-values with laws
				\begin{align*}
					p_1 \sim \textnormal{Unif}\bigl((0, 0.01) \cup (0.51, 1)\bigr), \quad p_2 \sim \textnormal{Unif}(0.01, 0.51),
				\end{align*}
				under the true distribution $P^*$ if it is in the null.
				Both these p-values are invalid:
				\begin{align*}
					P^*(p_1 \leq 0.01) = 0.02, \quad P^*(p_2 \leq 0.05) = 0.08.
				\end{align*}

				Now suppose we toss an independent fair coin and report $p_1$ on heads and $p_2$ on tails.
				Then the reported p-value is exactly $\textnormal{Unif}(0, 1)$, and hence valid.
				This means that we can knowingly select from two invalid p-values and report the selected value unchanged, yet obtain a valid p-value if we account for the selection mechanism.

				This phenomenon does not rely on external randomization.
				Indeed, we could compute both candidate p-values from all but one observation, and reserve that observation to choose between the candidates.
				This means the phenomenon concerns how a p-value aggregates different parts of the data.

				For e-values, the phenomenon does not happen.
				If two e-values $e_1$ and $e_2$ are both invalid under $P^*$ if it is in the null, an independent choice between them remains invalid:
				\begin{align*}
					\Ex^{P^*}[e] = w_1\Ex^{P^*}[e_1] + w_2\Ex^{P^*}[e_2] > 1, \quad w_1, w_2 \geq 0,\quad  w_1 + w_2 = 1.
				\end{align*}
			\end{example}

		\subsection{Valid evidence-based decision making}
			Our main technical contribution is to show how e-values emerge in \emph{every decision problem} with a notion of validity that satisfies three axioms.
			Here, we replace calibration to classical testing with an axiom that states invalidity is stable: a procedure that selects an invalid procedure with arbitrarily high probability must be invalid.
			We show the resulting axioms hold if and only if validity is characterized by an expected non-negative loss bound.
			We subsequently show how this enables post-hoc decision-making with e-values by transferring their validity to decisions, giving an axiomatic foundation to the pioneering work of \citet{grunwald2024beyond}.
			We also show how utility optimal e-values \citep{koning2024continuous, larsson2025numeraire} lead to utility-optimal decisions, and how calibration to classical testing emerges from this structure.
			
			Our second technical contribution is to study the mathematical structure underlying evidence-based post-hoc decision-making.
			The main conceptual insight is that this structure is purely order-theoretic.
			This enables us to generalize beyond e-values to preorder-valued evidence variables that represent validity of the decision variable, if evidence validity is monotone in the preorder.
			We subsequently show how recent closure results emerge \citep{xu2025bringing, koning2023post, koning2026measure, bai2026uniform}, and show how evidence variables should be chosen to enable optimal decision-making. 
			To highlight that this does not rely on e-values, we show how p-values also enable post-hoc decision-making.
			
			We apply our results to provide an axiomatic foundation for expectation-based error rates in multiple testing and show how the `e-closure principle' emerges, using rejection sets as decisions \citep{xu2025bringing}.
			Moreover, we show how the post-hoc choice of decision problem considered in \citet{grunwald2024beyond} can be recovered by viewing the problem itself as part of an abstract decision.
			We also recover the closure result of e-measures \citep{koning2026measure}.
						
			Finally, we motivate why we model non-rejection decisions across tests at different levels as the same decision.
			In addition, we give a separate numerical characterization of e-value validity using its averaging property.
			In Appendix \ref{sec:betting}, we relate our work to the testing by betting interpretation of the e-value \citep{shafer2021testing}.
			
		\subsection{Related literature}
			Independent selection between e-values, also known as `mixing', has received relatively little attention in the e-value literature \citep{wang2024testing}.
			Indeed, averaging is a much more prominent merging operation \citep{wang2025admissible, clerico2026geometric}, and even the `method of mixtures' refers to averaging instead of mixing \citep{ramdas2025hypothesis}.
			At the same time, mixing is a more fundamental operation, as it applies to arbitrary decision variables without requiring any arithmetic.
		
			To the best of our knowledge, the first modern connection between e-values and testing with data-dependent levels was made by \citet{wang2022false}, \citet{katsevich2020simultaneous} and \citet{xu2024post}, who find that e-values enable the post-hoc selection of the level in multiple testing, and in particular when testing a single hypothesis.

			This was developed into a generalization of the classical Neyman--Pearson framework to data-dependent levels in an expected loss setting by \citet{grunwald2024beyond}.
			\citet{koning2023post} subsequently showed the necessity of e-values for post-hoc testing, used the expected distortion ratio \eqref{ineq:expected_distortion_ratio} to define post-hoc validity, studied optimal post-hoc testing and introduced the post-hoc p-value.
			\citet{koning2024continuous} introduced the idea to view the e-value as a multi-decision generalization of a hypothesis test that rejects at some data-driven level, deriving an extension of the Neyman--Pearson lemma to e-values.
			\citet{chugg2026admissibility} study admissibility in the framework of \citet{grunwald2024beyond}.

			Post-hoc inference has already been applied in several settings, including conformal prediction \citep{gauthier2025values, gauthier2026backward, koning2025fuzzy, zhu2026beyond}, knockoffs \citep{fischer2025knockoffs}, and equivalence testing \citep{koobs2026equivalence}.
			Moreover, \citet{hartog2025family}, \citet{xu2025bringing} and \citet{koning2026measure} consider multiple testing with a data-driven level, and \citet{chugg2026post} study post-hoc inference in an asymptotic setting.
			
			The uncertainty over whether \eqref{ineq:expected_distortion_ratio} is the `right' extension of Type-I error to data-dependent levels is reflected in the work of \citet{gauthier2025values}, \citet{gauthier2026backward} and \citet{zhu2026beyond}, who instead focus on
			\begin{align}\label{ineq:expected_level}
				P(\textnormal{test falsely rejects hypothesis})
					\leq \Ex^P[\widetilde{\alpha}].
			\end{align}
			We feature this extension in Example \ref{exm:arithmetic_validity}.
			
			The debate over a right notion of validity has a historical precedent.
			Indeed \eqref{ineq:distortion} can be viewed as a special case of the conditional confidence approach of \citet{kiefer1976admissibility}.
			The expected-level bound \eqref{ineq:expected_level} corresponds, for binary loss, to the estimated-risk approach of \citet{berger1985statistical}. 
			The size-distortion bound \eqref{ineq:expected_distortion_ratio} was proposed as an alternative by \citet{vovk1993logic}.

	\section{Setup}\label{sec:data-dependent_levels}
		We use $\X$ to denote our sample space, and $P$ for a probability measure.
		
		Throughout, we compare tests across different significance levels.
		For this reason, it is important to treat rejections at different levels as different decisions.
		We capture these decisions in what we call an \emph{evidence space}, which is a totally ordered decision space $(\E, \precsim)$.
		We write $d \precsim d'$ when $d'$ represents evidence at least as strong as $d$, and $d \prec d'$ when the order is strict.
			
		We use this evidence space to define a level $\alpha \in (0, 1)$ test as a $\{n, r_\alpha\}$-valued random variable, where $n, r_\alpha \in \E$.
		Here, $r_\alpha$ represents the decision to reject at significance level $\alpha$.
		The decision $n$ represents a non-rejection.\footnote{In Section \ref{sec:distinguishing_non_rejections} we distinguish non-rejections at different levels.}
		We further assume that $n \prec r_{\alpha^+} \prec r_{\alpha^-}$, for every $\alpha^- < \alpha^+$, expressing that a rejection at a smaller level is a `stronger decision', and that a non-rejection is the weakest possible decision.

		To prepare for testing at data-dependent levels, we assume we have access to a family $\phi$ of tests $\phi = \{\phi(\alpha) : \alpha \in (0,1)\}$ across all levels $\alpha \in (0, 1)$, where $\phi(\alpha)$ is a $\{n, r_\alpha\}$-valued level $\alpha$ test.
		We say that $\phi$ is valid for $\alpha$ if the test $\phi(\alpha)$ is valid, in the sense that its Type-I error is at most $\alpha$.

		\begin{definition}[Classical validity]
			$\phi$ is valid for $\alpha$ if $P(\phi(\alpha) = r_\alpha) \leq \alpha$.
		\end{definition}

		We use $\widetilde{\alpha}$ to denote an arbitrary $(0,1)$-valued random level.
		Testing at this level produces the decision $\phi(\widetilde{\alpha})$.
		
		For a decision $d \in \E$, we write $\mu_d$ for the point mass at $d$, so that $\mu_d(A) = \ind{d \in A}$ for every measurable $A \subseteq \E$.
		For an $\E$-valued random variable $\delta$, we write $\mu_\delta$ for its law.

		Whenever randomization is needed, we silently enlarge the underlying probability space by auxiliary random variables that are independent of the data, and continue to write $P$ for the resulting probability measure.
		This allows both the test and the choice of the level to be randomized.

		Throughout, we impose the measurability conditions needed to make this random decision well defined.\footnote{Equip the reported decisions with the sigma-algebra generated by $\{n\}$ and $\{r_a : a\in B\}$ for every Borel set $B \subseteq (0,1)$. Assume that $a\mapsto r_a$ and its inverse on $\{r_a : a\in(0,1)\}$ are measurable, and that the reported decision is jointly measurable in the level and the underlying outcome. Whenever an index with law $\lambda$ is drawn independently, we use the product law $\lambda\otimes P$, with $P$ including any randomization used to generate the family.}

	\section{Axioms for data-dependent level validity}\label{sec:characterization}
		While it is straightforward to define a test with a data-dependent level, it is not obvious how to extend validity of tests from fixed to data-dependent levels.
		The key problem is that the object $\phi(\widetilde{\alpha})$ is no longer a test in the classical sense: its outcome is no longer $\{n,r_a\}$-valued for some fixed $a\in(0,1)$.
		Indeed, $\phi(\widetilde{\alpha})$ may produce a decision to reject at various different levels.
		
		To extend validity to data-dependent levels, we impose three simple axioms on a notion of validity.
		Our first axiom is that validity only depends on the distribution of the reported decision.
		We believe this axiom is quite natural, as it means that we can abstract away from the underlying sample space, and only consider the frequency at which each decision is made.
		
		\begin{testingaxiom}[Law invariance]\label{axiom:distribution}
			If $\phi(\widetilde{\alpha}) \overset{d}{=} \phi'(\widetilde{\alpha}')$ under $P$, then  $\phi$ is valid for $\widetilde{\alpha}$ if and only if $\phi'$ is valid for $\widetilde{\alpha}'$.
		\end{testingaxiom}

		The second axiom is that the notion of validity agrees with classical validity when the level is constant.
		This axiom ensures that the resulting notion of validity for data-dependent levels can be viewed as an extension of classical Type-I error validity.

		\begin{testingaxiom}[Extending classical testing]\label{axiom:classical_calibration}
			Suppose that $\widetilde{\alpha} \equiv a$ for some $a \in (0,1)$.
			Then $\phi$ is valid for $\widetilde{\alpha}$ if and only if $P(\phi(a) = r_a) \leq a$.
		\end{testingaxiom}

		The final axiom concerns preserving validity and invalidity under a data-independent choice of the level.
		Its interpretation is that if two tests $\phi(\alpha_1)$ and $\phi(\alpha_2)$ are both classically valid at their respective levels, then a procedure which selects one of the two before seeing the data must also be valid.
		Similarly, it states that a data-independent choice between two invalid tests must be invalid.

		\begin{testingaxiom}[Independent selection]\label{axiom:mixing}
			Suppose $\widetilde{\alpha}$ is independent of $\phi$.
			\begin{enumerate}[label=(\alph*), ref=\thetestingaxiom(\alph*), leftmargin=*, topsep=0.4em, itemsep=0.2em, parsep=0pt]
				\item\label{axiom:mixing_valid} If $\phi$ is valid for almost every constant level $\alpha$ in the distribution of $\widetilde{\alpha}$, then $\phi$ is valid for $\widetilde{\alpha}$.
				\item\label{axiom:mixing_invalid} If $\phi$ is invalid for almost every constant level $\alpha$ in the distribution of $\widetilde{\alpha}$, then $\phi$ is invalid for $\widetilde{\alpha}$.
			\end{enumerate}
		\end{testingaxiom}

		Theorem \ref{thm:characterizing_validity} shows that only the expected size distortion is compatible with these three axioms:
		\begin{align}\label{ineq:characterized_validity}
			\phi \textnormal{ is valid for } \widetilde{\alpha}
			\quad\Longleftrightarrow\quad
			\Ex^P\left[
				\frac{P(\phi(\widetilde{\alpha}) = r_{\widetilde{\alpha}} \mid \widetilde{\alpha})}
					 {\widetilde{\alpha}}
			\right]
			\leq 1.
		\end{align}
		Its proof is given in Appendix~\ref{sec:proof_characterizations}.
		Appendix~\ref{sec:testing_examples} contains two examples that illustrate these axioms.

		\begin{theorem}\label{thm:characterizing_validity}
			The criterion \eqref{ineq:characterized_validity} is the unique notion of validity that satisfies Axioms \ref{axiom:distribution}-\ref{axiom:mixing}.
		\end{theorem}
		
		Corollary \ref{cor:composite_hypotheses} shows that the result directly extends to a composite hypothesis $\P$, which is a collection of probability measures.
		Such a composite hypothesis represents the problem that we do not know which $P \in \P$ generated the data.
		To ensure validity under the true data-generating process, it is therefore standard to define validity for $\P$ as validity for every $P \in \P$.

		\begin{compositeassumption}[Composite validity]\label{assumption:composite_validity}
			A procedure is valid under $\P$ if and only if it is valid under every $P \in \P$.
		\end{compositeassumption}
		
		\begin{corollary}[Composite hypotheses]\label{cor:composite_hypotheses}
			Suppose Assumption~\ref{assumption:composite_validity} holds, and impose Axioms~\ref{axiom:distribution}-\ref{axiom:mixing} separately for every $P\in \P$.
			Then $\phi$ is valid for $\widetilde{\alpha}$ under $\P$ if and only if
			\begin{align*}
				\Ex^P\left[\frac{P(\phi(\widetilde{\alpha}) = r_{\widetilde{\alpha}} \mid \widetilde{\alpha})}
						 {\widetilde{\alpha}}\right]
					\leq 1, \textnormal{ for every } P \in \P.
			\end{align*}
		\end{corollary}
		\begin{proof}
			Apply Theorem~\ref{thm:characterizing_validity} under every $P\in \P$ and use Assumption~\ref{assumption:composite_validity}.
		\end{proof}
		
		\begin{remark}[Link between size distortion and expected loss]\label{rmk:link_size_distortion_expected_loss}
			A mathematically convenient way to write the expected size distortion \eqref{ineq:characterized_validity} is in an expected loss form:
			\begin{align*}
				\Ex^P[L(\phi(\widetilde{\alpha}))] \leq 1,
			\end{align*}
			for $L(n) = 0$ and $L(r_\alpha) = 1/\alpha$.
			Here $L$ measures loss under the null hypothesis, so stronger rejections incur greater loss.
			Indeed, by the law of iterated expectations
			\begin{align*}
				\Ex^P\left[\frac{P(\phi(\widetilde{\alpha})=r_{\widetilde{\alpha}}\mid\widetilde{\alpha})}
					 {\widetilde{\alpha}}\right]
					&= \Ex^P\left[\Ex^P\left[\frac{\ind{\phi(\widetilde{\alpha})=r_{\widetilde{\alpha}}}}
					 {\widetilde{\alpha}} \middle|\ \widetilde{\alpha}\ \right] \right]
					\\
					&= \Ex^P\left[\frac{\ind{\phi(\widetilde{\alpha})=r_{\widetilde{\alpha}}}}
					 {\widetilde{\alpha}}\right]
					= \Ex^P[L(\phi(\widetilde{\alpha}))].
			\end{align*}
		\end{remark}
		
	\section{Axioms for e-validity and p-validity}\label{sec:e-value}
		\subsection{Evidence variables}
			Abstracting away from evaluating a test at a data-dependent level, we consider generic $\E$-valued random variables, which we call \emph{evidence variables} \citep{koning2024continuous, koning2023post}.
			We denote such evidence variables by $\e$, and call their outcomes \emph{evidence values}.
			A test at a data-dependent level $\phi(\widetilde{\alpha})$ is an evidence variable.
			
			\begin{definition}\label{dfn:e-value_fundamental}
				An evidence variable is a $\E$-valued random variable.
			\end{definition}
			
			Evidence variables can be viewed as a unifying abstraction of tests, p-variables and e-variables for $\E = \{n\} \cup \{r_\alpha, \alpha \in (0, 1)\}$.
			Indeed, a $\{n, r_\alpha\}$-valued evidence variable is simply a level $\alpha$ test.
			Moreover, if we interpret $e = 1/\alpha$ and $p = \alpha$ as the decision $\e = r_\alpha$, and $e = 0$ and $p = 1$ as the decision $\e = n$, then we can view both e-variables and p-variables as different numerical representations of an evidence variable.
			Here and below, e-variables and p-variables denote random variables, while e-values and p-values denote their realized numerical values.
			Validity is a separate property: both types of variable may be valid or invalid.
			
			At this abstract evidence variable-level, e-variables and p-variables are distinguished by their validity criterion.
			For `e-validity', the map $L(n) = 0$, $L(r_\alpha) = 1/\alpha$ converts the evidence variable into its numerical representation $e = L(\e)$.
			The e-validity criterion $\Ex^P[e] \leq 1$ can therefore be interpreted as
			\begin{align}\label{iff:e-value}
				\e \textnormal{ is valid} \iff \Ex^P[L(\e)] \leq 1.
			\end{align}
			P-validity appears as the choice
			\begin{align}\label{iff:p-value}
				\e \textnormal{ is valid} \iff P(\e \succsim r_\alpha) \leq \alpha, \textnormal{ for every } \alpha \in (0, 1),
			\end{align}
			which can be numerically represented as $P(p \leq \alpha) \leq \alpha$, for every $\alpha \in (0, 1)$.
			Both criteria agree with classical validity for binary tests.
			
			Given this unifying perspective on e-variables and p-variables, the `p versus e' reduces to the question:
			\begin{quote}
				\emph{what notion of validity should we use for an evidence variable?}	
			\end{quote}
			In this section, we give axiomatic characterizations of both, and subsequently use these axioms to compare the two.
			
		\subsection{Characterizing e-validity}
			We present three axioms that characterize e-validity \eqref{iff:e-value}, which parallel the axioms presented in Section \ref{sec:characterization}.
			As before, the first axiom is that validity should only depend on the distribution of the evidence variable.
			In Appendix \ref{sec:conditionality}, we show that this axiom is equivalent to the conditionality principle of \citet{birnbaum1962foundations} applied to validity.

			\begin{evalueaxiom}[Law invariance]\label{axiom:evalue_distribution}
				If $\e \overset{d}{=} \e'$ under $P$, then $\e$ is valid if and only if $\e'$ is valid.
			\end{evalueaxiom}

			The second axiom similarly calibrates validity of evidence variables to classical level $\alpha$ tests, which appear as binary $\{n, r_\alpha\}$-valued evidence variables.

			\begin{evalueaxiom}[Extending classical testing]\label{axiom:evalue_calibration}
				Suppose that $\e$ is a level $\alpha$ test.
				Then $\e$ is valid if and only if $P(\e = r_\alpha) \leq \alpha$.
			\end{evalueaxiom}
		
			The third axiom states that a procedure which selects (`mixes') between evidence variables independently of the data must be valid if each possible choice is valid, and invalid if each possible choice is invalid.
			Another way of stating this is that revealing which evidence variable was selected should not change the validity status if the possible selections were either all valid or all invalid.
			
			It is analogous to Axiom \ref{axiom:mixing}, where the data-independent choice concerned the level $\widetilde{\alpha}$, selecting a test from a family of tests.
			Axiom \ref{axiom:evalue_mixing} concerns selecting between arbitrary evidence variables.
			Here, a mixture of a family $(\e_t)_{t\in T}$ is $\e_I$, where $I$ is an independently drawn index in $T$.

			\begin{evalueaxiom}[Independent selection]\label{axiom:evalue_mixing}
				\leavevmode
				\begin{enumerate}[label=(\alph*), ref=\theevalueaxiom(\alph*), leftmargin=*, topsep=0.4em, itemsep=0.2em, parsep=0pt]
					\item\label{axiom:evalue_mixing_valid} Mixtures of valid evidence variables are valid evidence variables.
					\item\label{axiom:evalue_mixing_invalid} Mixtures of invalid evidence variables are invalid evidence variables.
				\end{enumerate}
			\end{evalueaxiom}
						
			Theorem \ref{thm:e-value} shows that these axioms pin down the notion of validity \eqref{iff:e-value}.
			This can be extended to composite hypotheses by applying the axioms and result to every $P \in \P$, as in Corollary \ref{cor:composite_hypotheses}.
			
			We view this result as stating that e-validity, with the interpretation that the e-value $e = 1/\alpha$ is a rejection at level $\alpha$, is the unique notion of validity that extends classical testing under law invariance and our independent selection axiom.
			A proof sketch is given below, which is illustrated in Figure~\ref{fig:evidence-axioms}.
			
			\begin{theorem}\label{thm:e-value}
				If validity satisfies Axioms \ref{axiom:evalue_distribution}-\ref{axiom:evalue_mixing}, then
				\begin{align*}
					\e\textnormal{ is valid}
						\iff \Ex^P[L(\e)] \leq 1,
				\end{align*}
				where $L(n) = 0$ and $L(r_\alpha) = 1/\alpha$.
			\end{theorem}
			
			\begin{proof}[Proof sketch]
				Consider a $\{n, r_a, r_b\}$-valued evidence variable $\e$, with $0 < a < b < 1$.
				Write $x = P(\e = r_a)$ and $y = P(\e = r_b)$.
				
				By Axiom \ref{axiom:evalue_distribution}, validity depends only on the point $(x, y)$ in the probability triangle $x, y \geq 0$, $x + y \leq 1$.
				Axiom \ref{axiom:evalue_calibration} determines validity on its axes:
				\begin{align*}
				    (x,0)\textnormal{ valid}\iff x\leq a,
				    \quad
				    (0,y)\textnormal{ valid}\iff y\leq b.
				\end{align*}
				
				Mixtures correspond to convex combinations of points.
				Axiom \ref{axiom:evalue_mixing_valid} therefore makes the entire triangle with vertices $(0,0)$, $(a, 0)$, and $(0, b)$ valid.
				This is precisely the region $x/a + y/b \leq 1$.
				Every point outside this triangle is a mixture of classically invalid points on the two axes, and is therefore invalid by Axiom \ref{axiom:evalue_mixing_invalid}.
				Hence, $\e$ is valid if and only if $x/a + y/b \equiv \Ex^P[L(\e)] \leq 1$.
			\end{proof}

			\begin{figure}[!ht]
				\centering
				\resizebox{\linewidth}{!}{
\begingroup%
\definecolor{axiomValid}{HTML}{126B60}%
\definecolor{axiomValidFill}{HTML}{DDEDE7}%
\definecolor{axiomInvalid}{HTML}{B55332}%
\definecolor{axiomInvalidFill}{HTML}{F7E3D9}%
\definecolor{axiomUnknown}{HTML}{EDF0F2}%
\definecolor{axiomGray}{HTML}{65707A}%
\definecolor{axiomOutline}{HTML}{A7AFB5}%
\def\axiomA{0.40}%
\def\axiomB{0.66}%
\begin{tikzpicture}[
  x=1cm,y=1cm,
  font=\fontsize{12}{14}\selectfont,
  line cap=round,line join=round,
  axiom vertex/.style={font=\fontsize{12}{14}\selectfont,inner sep=2pt},
  axiom title/.style={font=\fontsize{11}{13}\selectfont\bfseries,align=center},
  axiom region/.style={font=\fontsize{10}{12}\selectfont,inner sep=0pt,align=center},
  axiom cutoff/.style={circle,fill=axiomValid,draw=axiomValid,
    inner sep=0pt,minimum size=4.2pt},
  axiom invalid point/.style={circle,fill=axiomInvalid,draw=axiomInvalid,
    inner sep=0pt,minimum size=3.6pt}
]

\node at (6.75,6.03) {\ref{axiom:evalue_distribution}: law invariance};
\node at (6.75,5.50)
  {$(1-x-y)\mu_n+x\mu_{r_a}+y\mu_{r_b},
    \qquad x,y\ge0,\quad x+y\le1$};
\draw[axiomOutline,line width=0.35pt] (-0.30,5.08)--(13.60,5.08);

\node at (2.00,4.55) {(a) \ref{axiom:evalue_calibration} Binary calibration};
\node at (6.75,4.55) {(b) \ref{axiom:evalue_mixing_valid} Valid mixtures};
\node at (11.50,4.55) {(c) \ref{axiom:evalue_mixing_invalid} Invalid mixtures};

\foreach \panel/\offset in {1/0.20,2/4.95,3/9.70}{
  \begin{scope}[shift={(\offset,0.45)},x={(3.6cm,0cm)},y={(1.8cm,3.117691cm)}]
    \coordinate (N) at (0,0);
    \coordinate (R) at (1,0);
    \coordinate (S) at (0,1);
    \coordinate (A) at (\axiomA,0);
    \coordinate (B) at (0,\axiomB);

    \fill[axiomUnknown] (N)--(R)--(S)--cycle;
    \ifnum\panel=3
      \fill[axiomInvalidFill] (A)--(R)--(S)--(B)--cycle;
    \fi
    \ifnum\panel>1
      \fill[axiomValidFill] (N)--(A)--(B)--cycle;
    \fi

    \draw[axiomOutline,line width=1.5pt] (R)--(S);
    \ifnum\panel=3
      \draw[axiomInvalid,line width=1.5pt] (R)--(S);
    \fi
    \draw[axiomInvalid,line width=1.5pt] (A)--(R);
    \draw[axiomInvalid,line width=1.5pt] (B)--(S);
    \draw[axiomValid,line width=1.5pt] (N)--(A);
    \draw[axiomValid,line width=1.5pt] (N)--(B);

    \ifnum\panel>1
      \draw[axiomValid,line width=1.5pt] (A)--(B);
      \node[axiom region,text=axiomValid] at (0.16,0.14) {Valid};
    \fi
    \ifnum\panel=2
      \node[axiom cutoff,minimum size=3.6pt]
        at ({0.55*\axiomA},{0.45*\axiomB}) {};
    \fi
    \ifnum\panel=3
      \draw[axiomInvalid,line width=0.65pt,dash pattern=on 2pt off 2pt]
        (0.88,0)--(0,0.94);
      \node[axiom invalid point] at (0.88,0) {};
      \node[axiom invalid point] at (0,0.94) {};
      \node[axiom invalid point] at (0.44,0.47) {};
      \node[axiom region,text=axiomInvalid] at (0.50,0.18) {Invalid};
    \fi

    \node[axiom cutoff] at (A) {};
    \node[axiom cutoff] at (B) {};
    \node[text=axiomValid,below=3pt,inner sep=1pt] at (A) {$a$};
    \node[text=axiomValid,left=3pt,inner sep=1pt] at (B) {$b$};
    \node[axiom vertex,below left=2pt] at (N) {$\mu_n$};
    \node[axiom vertex,below=3pt] at (R) {$\mu_{r_a}$};
    \node[axiom vertex,above=4pt] at (S) {$\mu_{r_b}$};
  \end{scope}
}

\node[align=center] at (2.00,-0.72)
  {$\begin{gathered}
     (x,0)\text{ valid}\iff x\le a\\[-1pt]
     (0,y)\text{ valid}\iff y\le b
   \end{gathered}$};
\node at (6.75,-0.72)
  {$\dfrac{x}{a}+\dfrac{y}{b}\le1\ \Longrightarrow\ \text{valid}$};
\node at (11.50,-0.72)
  {$\dfrac{x}{a}+\dfrac{y}{b}>1\ \Longrightarrow\ \text{invalid}$};

\end{tikzpicture}%
\endgroup%
}
				\caption{\ref{axiom:evalue_distribution} represents laws as points, \ref{axiom:evalue_calibration} calibrates the binary edges, \ref{axiom:evalue_mixing_valid} fills the valid (green) region and \ref{axiom:evalue_mixing_invalid} the invalid (orange) region. Gray denotes laws not yet classified.}
				\label{fig:evidence-axioms}
			\end{figure}
			
			\begin{remark}[The domain of $\alpha$ and $e$]\label{rmk:alpha_domain}
			Because $L(r_\alpha) = 1/\alpha$ for $\alpha \in (0, 1)$, the numerical e-variable representation $e = L(\e)$ of the evidence variable takes values only in $\{0\} \cup (1,\infty)$.
				The full numerical range $[0,\infty]$ can be obtained by adding decisions $r_\alpha$ for $\alpha \in \{0\} \cup [1,\infty]$, with $L(r_\alpha) = 1/\alpha$, $1/0 = \infty$, and $r_\infty = n$.
				This extends beyond classical testing and is not characterized by the axioms above (see Section \ref{sec:recovering_classical_calibration}).
			\end{remark}	
	
		\subsection{Characterizing p-validity}\label{sec:p-values}
			We now present axioms that characterize p-validity \eqref{iff:p-value}.
			The idea is to replace independent selection (\ref{axiom:evalue_mixing}) with monotonicity and comonotone aggregation axioms.
			
			Monotonicity is a natural axiom, and can be interpreted as stating that discarding evidence is allowed.
			For example, if we find a rejection at level 4.7\%, it implies that it is also valid to claim a rejection at level 5\%.
			
			\begin{evalueaxiom}[Monotonicity]\label{axiom:monotonicity}
				If $\e$ is valid and $\e' \precsim \e$ pointwise, then $\e'$ is valid.
			\end{evalueaxiom}
			
			Under Axioms \ref{axiom:evalue_distribution}, \ref{axiom:evalue_calibration} and \ref{axiom:monotonicity}, validity of an evidence variable implies \eqref{iff:p-value}.
			We obtain a converse with validity under comonotone aggregation.
			An interpretation of comonotonicity is that the evidence variables are increasing in a common statistic.
			
			\begin{definition}
				A family $(\e_t)_{t \in T}$ of evidence variables is comonotone if, for every $i, j \in T$ and outcomes $x, y \in \X$, $\e_i(x) \succ \e_i(y)$ implies $\e_j(x) \succsim \e_j(y)$.
			\end{definition}
			
			\begin{evalueaxiom}[Comonotone aggregation]\label{axiom:comonotone}
				Countable suprema of valid comonotone evidence variables are valid whenever they exist in $\E$.
			\end{evalueaxiom}
			
			Theorem~\ref{thm:p-value} gives an axiomatic characterization of p-validity.
			Its proof uses the same nested binary-test construction as used by \citet{blierwong2026improved} in the representation of valid p-variables using comonotonic valid e-variables.
			
			\begin{theorem}\label{thm:p-value}
				If validity satisfies Axioms \ref{axiom:evalue_distribution}, \ref{axiom:evalue_calibration}, \ref{axiom:monotonicity} and \ref{axiom:comonotone}, then
				\begin{align*}
					\e \textnormal{ is valid } \iff P(\e \succsim r_\alpha) \leq \alpha, \textnormal{ for every } \alpha \in (0, 1).
				\end{align*}
			\end{theorem}
			\begin{proof}
				For every $\alpha \in (0, 1)$, define the binary evidence variable
				\begin{align*}
					\e_\alpha
						=
						\begin{cases}
							r_\alpha,  &\e\succsim r_\alpha, \\
							n, &\textnormal{otherwise}.
						\end{cases}
				\end{align*}	
				Since $\e_\alpha \precsim \e$, \ref{axiom:monotonicity} makes $\e_\alpha$ valid whenever $\e$ is valid.
				Axiom \ref{axiom:evalue_calibration} then gives $P(\e \succsim r_\alpha) = P(\e_\alpha = r_\alpha) \leq \alpha$, for every $\alpha \in (0, 1)$ (Figure~\ref{fig:calibration-closure-comparison}(a)).
				
				Conversely, suppose that these inequalities hold.
				By \ref{axiom:evalue_calibration}, every $\e_\alpha$ is valid.
				These evidence variables are comonotone and satisfy $\e = \sup_{\alpha \in \mathbb{Q} \cap (0, 1)} \e_\alpha$.
				Axiom \ref{axiom:comonotone} therefore makes $\e$ valid (Figure~\ref{fig:calibration-closure-comparison}(b)).

				The criterion \eqref{iff:p-value} itself satisfies \ref{axiom:evalue_distribution}, \ref{axiom:evalue_calibration} and \ref{axiom:monotonicity}.
				For a countable comonotone family of p-valid evidence variables $\e_i$, the events $\{\e_i \succ r_\beta\}$ are nested, so their union $\{\sup_i \e_i \succ r_\beta\}$ has probability at most $\beta$.
				Letting $\beta \downarrow \alpha$ gives $P(\sup_i \e_i \succsim r_\alpha) \leq \alpha$, proving \ref{axiom:comonotone}.
			\end{proof}

		\subsection{Comparing e-validity and p-validity}\label{sec:e-vs-p}
			The axiomatizations give us a concrete way to compare the two notions of validity.
			Beyond Axioms \ref{axiom:evalue_distribution} and \ref{axiom:evalue_calibration}, both satisfy monotonicity (\ref{axiom:monotonicity}) and both preserve validity under independent selection (\ref{axiom:evalue_mixing_valid}).
			Moreover, e-validity implies p-validity, so that every e-valid evidence variable is also p-valid.
			
			The notions of validity disagree on evidence variables that are p-valid but not e-valid, illustrated by the gray region in Figure~\ref{fig:calibration-closure-comparison}(c).
			The following result characterizes the gap in terms of independent selection.
			It shows that every additional evidence variable admitted by p-validity causes a failure of Axiom \ref{axiom:evalue_mixing_invalid}: its law can be reproduced by independently selecting among tests that are all invalid, yet the resulting evidence variable is p-valid.
			
			\begin{proposition}\label{prp:invalid_binary_mixture}
				Every p-valid but e-invalid evidence variable has the same distribution as an independent mixture of classically invalid binary tests.
			\end{proposition}
			
			The result means that the `p versus e'-discussion reduces to Axiom \ref{axiom:evalue_mixing_invalid}: \emph{should independent selection between invalid evidence variables be able to produce a valid evidence variable?}

			\begin{figure}[!ht]
				\centering
				\resizebox{\linewidth}{!}{
\begingroup%
\definecolor{ccValid}{HTML}{126B60}%
\definecolor{ccValidFill}{HTML}{DDEDE7}%
\definecolor{ccInvalid}{HTML}{B55332}%
\definecolor{ccInvalidFill}{HTML}{F7E3D9}%
\definecolor{ccUnknown}{HTML}{EDF0F2}%
\definecolor{ccGray}{HTML}{65707A}%
\def\ccA{0.40}%
\def\ccB{0.66}%
\begin{tikzpicture}[
  x=1cm,y=1cm,
  font=\fontsize{12}{14}\selectfont,
  line cap=round,line join=round,
  cc vertex/.style={font=\fontsize{12}{14}\selectfont,inner sep=2pt},
  cc region/.style={font=\fontsize{10}{12}\selectfont,inner sep=0pt,align=center},
  cc cutoff/.style={circle,fill=ccValid,draw=ccValid,
    inner sep=0pt,minimum size=4.2pt},
  cc invalid point/.style={circle,fill=ccInvalid,draw=ccInvalid,
    inner sep=0pt,minimum size=3.6pt}
]

\node at (1.80,4.55) {(a) \ref{axiom:evalue_distribution} + \ref{axiom:evalue_calibration} + \ref{axiom:monotonicity}};
\node at (6.55,4.55) {(b) Panel (a) + \ref{axiom:comonotone}};
\node at (11.30,4.55) {(c) Panel (a) + \ref{axiom:evalue_mixing_valid}};

\foreach \panel/\offset in {1/0,2/4.75,3/9.50}{
  \begin{scope}[shift={(\offset,0.45)},x={(3.6cm,0cm)},y={(1.8cm,3.117691cm)}]
    \coordinate (N) at (0,0);
    \coordinate (R) at (1,0);
    \coordinate (S) at (0,1);
    \coordinate (A) at (\ccA,0);
    \coordinate (B) at (0,\ccB);
    \coordinate (C) at (\ccA,{\ccB-\ccA});
    \coordinate (D) at (0,\ccA);

    \fill[ccUnknown] (N)--(R)--(S)--cycle;
    \fill[ccInvalidFill] (A)--(R)--(S)--(B)--(C)--cycle;
    \ifnum\panel=1
      \fill[ccValidFill] (N)--(A)--(D)--cycle;
    \else\ifnum\panel=2
      \fill[ccValidFill] (N)--(A)--(C)--(B)--cycle;
    \else
      \fill[ccValidFill] (N)--(A)--(B)--cycle;
    \fi\fi

    \draw[ccInvalid,line width=1.5pt] (R)--(S);
    \draw[ccInvalid,line width=1.5pt] (A)--(R);
    \draw[ccInvalid,line width=1.5pt] (B)--(S);
    \draw[ccValid,line width=1.5pt] (N)--(A);
    \draw[ccValid,line width=1.5pt] (N)--(B);
    \ifnum\panel=2
      \draw[ccValid,line width=1.5pt] (A)--(C)--(B);
    \else
      \draw[ccGray,line width=1pt] (A)--(C)--(B);
      \ifnum\panel=1
        \draw[ccValid,line width=1.5pt] (A)--(D);
      \else
        \draw[ccValid,line width=1.5pt] (A)--(B);
      \fi
    \fi
    \ifnum\panel=1
      \node[cc region,text=ccValid] at (0.15,0.10) {Valid};
    \else
      \node[cc region,text=ccValid] at (0.16,0.14) {Valid};
    \fi
    \node[cc region,text=ccInvalid] at (0.64,0.14) {Invalid};

    \ifnum\panel=3
      \coordinate (U) at ({1.2*\ccA},0);
      \coordinate (V) at (0,{1.2*\ccB});
      \coordinate (M) at ($(U)!0.375!(V)$);
      \draw[ccGray,line width=0.85pt,dash pattern=on 3pt off 3pt]
        (U)--(V);
      \node[cc invalid point] at (U) {};
      \node[cc invalid point] at (V) {};
      \node[circle,fill=ccGray,draw=ccGray,inner sep=0pt,minimum size=3.6pt] at (M) {};
    \fi

    \node[cc cutoff] at (A) {};
    \node[cc cutoff] at (B) {};
    \node[text=ccValid,below=3pt,inner sep=1pt] at (A) {$a$};
    \node[text=ccValid,left=3pt,inner sep=1pt] at (B) {$b$};
    \node[cc vertex,below left=2pt] at (N) {$\mu_n$};
    \node[cc vertex,below=3pt] at (R) {$\mu_{r_a}$};
    \node[cc vertex,above=4pt] at (S) {$\mu_{r_b}$};
  \end{scope}
}

\node at (1.80,-0.52)
  {$\text{valid}\ \Longrightarrow\
    \begin{gathered}x\le a\\[-1pt]x+y\le b\end{gathered}$};
\node at (6.55,-0.52)
  {$\text{valid}\ \Longleftrightarrow\
    \begin{gathered}x\le a\\[-1pt]x+y\le b\end{gathered}$};
\node at (11.30,-0.52)
  {$\begin{gathered}
      \text{e-valid}\ \Longrightarrow\ \text{valid}\\[-1pt]
      \text{valid}\ \Longrightarrow\ \text{p-valid}
    \end{gathered}$};
\end{tikzpicture}%
\endgroup%
}
				\caption{Continuing from panel (a) in Figure~\ref{fig:evidence-axioms}, panel (a) shows that adding \ref{axiom:monotonicity} forces some laws to be valid. Panel (b) adds \ref{axiom:comonotone}, characterizing p-validity. Panel (c) instead adds \ref{axiom:evalue_mixing_valid} to panel (a), so that the gray region contains p-valid but e-invalid laws. Its dashed segment joins two invalid binary laws whose mixture lies in this gap.}
				\label{fig:calibration-closure-comparison}
			\end{figure}
			\FloatBarrier
	
	\section{Validity in arbitrary decision problems}
		In this section, we characterize validity in arbitrary decision problems.
		We show that non-negative expected-loss bounds arise if we replace the calibration to classical testing by  an axiom that limits how invalid decisions can be compensated.
		We subsequently show how e-values emerge as labels of equivalence classes of decisions, and show how these labels can be interpreted from a classical testing perspective.
		
		\subsection{Stability of invalidity}
			We consider an arbitrary decision space $\D$ without any order structure, and call a $\D$-valued random variable a \emph{decision variable}.\footnote{\label{fn:validity_measurability}Throughout this section, we equip $\D$ with a sigma-algebra and take decision variables to be measurable. Mixtures use families $(\e_t)_{t\in T}$ on a common probability space, indexed by a measurable space $(T,\mathcal T)$, with $(t,\omega)\mapsto\e_t(\omega)$ jointly measurable. For every such family, assume that $\{t\in T:\e_t\textnormal{ is valid}\}\in\mathcal T$.}
			
			We consider replacing the calibration to classical testing (Axiom \ref{axiom:evalue_calibration}) by a \emph{stability of invalidity} axiom.			
			To present the axiom, we use $q_\delta$ to denote the supremal probability with which a decision variable $\delta$ can be selected as part of a valid procedure:
			\begin{align*}
				q_\delta
					:= \sup\{q \in [0, 1] : q \mu_\delta + (1 - q) \nu, \textnormal{ is valid for at least one law } \nu\},
			\end{align*}
			where $\mu_\delta$ is the law of $\delta$, $\sup \emptyset = 0$.
			A valid decision variable $\delta$ has $q_\delta = 1$.
			Here, we rely on Axiom \ref{axiom:evalue_distribution}, which allows us to speak of `valid laws'.
			
			\begin{evalueaxiom}[Stability of invalidity]\label{axiom:stability_of_invalidity}
				If $\delta$ is invalid, then $q_\delta < 1$.
			\end{evalueaxiom}
			
			In words, the axiom states that a procedure which selects an invalid procedure with arbitrarily high probability cannot be valid (see Example \ref{exm:unstable_validity}).
			In Figure~\ref{fig:arbitrary_calibration_geometry}(b) we give a geometric interpretation: shrinking the entire probability simplex towards the distribution of an invalid decision variable must leave a sufficiently small copy entirely invalid.
			
			\begin{example}[Violating Axiom \ref{axiom:stability_of_invalidity}]\label{exm:unstable_validity}
				The notion of validity $\Ex^P[\delta]\ge1$ on $\D = [0, \infty]$ satisfies law invariance and independent selection but violates Axiom \ref{axiom:stability_of_invalidity}.
				Indeed, it declares $\delta \equiv 0$ invalid, while declaring $\delta \sim (1-1/a) \mu_0 + (1/a) \mu_a$ valid for every $a \geq 2$.
				Letting $a\to\infty$ shows that $q_{\delta \equiv 0}=1$.
				This means that the invalidity of $\mu_0$ can be compensated by an arbitrarily small probability $1/a$ on $\mu_a$.
			\end{example}
			
		\subsection{Expected-loss representation of validity}
			Theorem~\ref{thm:expected_loss} shows that the validity of a decision variable $\delta$ can be represented as a non-negative expected loss bound under Axioms \ref{axiom:evalue_distribution}, \ref{axiom:evalue_mixing} and \ref{axiom:stability_of_invalidity}.
			A consequence is that validity only depends on the numerical loss values $L_P(d)$ assigned to every decision $d$ and the frequency at which they are made.
			
			\begin{theorem}[Expected-loss]\label{thm:expected_loss}
				Axioms \ref{axiom:evalue_distribution}, \ref{axiom:evalue_mixing} and \ref{axiom:stability_of_invalidity} hold if and only if there exists a measurable loss function $L_P : \D \to [0,\infty]$ such that every decision variable $\delta$ satisfies
				\begin{align*}
					\delta \textnormal{ is valid}
					\quad\iff\quad
					\Ex^P[L_P(\delta)] \leq 1.
				\end{align*}
				If $0 < q_{\delta_0} < 1$ for some decision variable $\delta_0$, there is a unique loss $L_P$ satisfying $\inf_{d\in\mathcal D} L_P(d) = 0$.
			\end{theorem}
			\begin{proof}[Proof sketch of Theorem~\ref{thm:expected_loss}]
				By law invariance, validity on each finite set of decisions is a subset of its coefficient simplex.
				Independent selection makes both regions convex (Figure~\ref{fig:arbitrary_calibration_geometry}(a)), while stability of invalidity makes the valid region closed (Figure~\ref{fig:arbitrary_calibration_geometry}(c)).
				A separating hyperplane argument then gives an expected-value bound for an auxiliary score $g$.
				Stability of invalidity also bounds $g$ from below: otherwise an invalid reference law could be mixed into a valid law by a tiny probability on a sufficiently extreme valid law (Figure~\ref{fig:arbitrary_calibration_geometry}(d)).
				Shifting and rescaling then gives a nonnegative loss $L$.
				The independent selection axiom extends the representation from finitely supported laws to arbitrary laws.
			\end{proof}
				
			\begin{figure}[!htbp]
				\centering
				\resizebox{\linewidth}{!}{
\begingroup
\definecolor{e5Valid}{HTML}{126B60}
\definecolor{e5ValidFill}{HTML}{DDEDE7}
\definecolor{e5Invalid}{HTML}{B55332}
\definecolor{e5InvalidFill}{HTML}{F7E3D9}
\definecolor{e5Cap}{HTML}{DDA78E}
\definecolor{e5Gray}{HTML}{65707A}
\def\eFiveH{0.866025404}
\begin{tikzpicture}[
  x=1cm,y=1cm,font=\fontsize{12}{14}\selectfont,
  line cap=round,line join=round,
  vertex/.style={inner sep=2pt},
  e5 point/.style={circle,inner sep=0pt,minimum size=3.7pt},
  region/.style={font=\fontsize{10}{12}\selectfont,inner sep=0pt},
  small label/.style={font=\fontsize{11}{13}\selectfont,inner sep=1pt}
]
\node at (1.75,3.88) {(a) \ref{axiom:evalue_mixing} Independent selection};
\node at (5.95,3.88) {(b) \ref{axiom:stability_of_invalidity}};
\node at (10.15,3.88) {(c) Boundary};
\node at (14.35,3.88) {(d) Uniform in $a$};

\foreach \panel/\offset in {1/0,2/4.2,3/8.4}{
  \begin{scope}[shift={(\offset,0)},x=3.5cm,y=3.5cm]
    \coordinate (A) at (0,0);
    \coordinate (B) at (1,0);
    \coordinate (C) at (.5,\eFiveH);
    \coordinate (U) at (.25,{.5*\eFiveH});
    \coordinate (V) at (.5,0);

    \fill[e5InvalidFill] (A)--(B)--(C)--cycle;
    \fill[e5ValidFill] (A)--(V)--(U)--cycle;
    \draw[e5Valid,line width=1.5pt] (A)--(V);
    \draw[e5Invalid,line width=1.5pt] (V)--(B)--(C);
    \draw[e5Valid,line width=1.5pt] (A)--(U);
    \draw[e5Invalid,line width=1.5pt] (U)--(C);

    \draw[e5Gray,line width=1pt,dash pattern=on 2.4pt off 2pt] (U)--(V);
    \node[e5 point,fill=e5Gray] at (U) {};
    \node[e5 point,fill=e5Gray] at (V) {};

    \ifnum\panel=1
      \draw[e5Valid,line width=.85pt] (.10,.10)--(.35,.08);
      \node[e5 point,fill=e5Valid] at (.10,.10) {};
      \node[e5 point,fill=e5Valid] at (.35,.08) {};
      \draw[e5Invalid,line width=.85pt] (.43,.21)--(.60,.57);
      \node[e5 point,fill=e5Invalid] at (.43,.21) {};
      \node[e5 point,fill=e5Invalid] at (.60,.57) {};
      \node[region,text=e5Valid] at (.265,.18) {Valid};
      \node[region,text=e5Invalid] at (.65,.30) {Invalid};
    \fi

    \ifnum\panel=2
      \coordinate (M) at (.55,{.5*\eFiveH});
      \coordinate (KA) at (.4125,{.375*\eFiveH});
      \coordinate (KB) at (.6625,{.375*\eFiveH});
      \coordinate (KC) at (.5375,{.625*\eFiveH});
      \foreach \P in {A,B,C}{
        \draw[e5Gray!45,line width=.55pt,dash pattern=on 2pt off 2pt] (M)--(\P);
      }
      \filldraw[fill=e5Cap,draw=e5Invalid,line width=1.2pt] (KA)--(KB)--(KC)--cycle;
      \foreach \P in {KA,KB,KC}{\node[e5 point,fill=e5Invalid,minimum size=2.6pt] at (\P) {};}
      \node[e5 point,fill=e5Invalid] at (M) {};
      \node[small label,below left=1pt,text=e5Invalid] at (M) {$\mu$};
      \node[small label,text=e5Invalid,below right=3pt] at (KB) {$q$};
    \fi

    \ifnum\panel=3
      \coordinate (M) at (.375,{.25*\eFiveH});
      \coordinate (KA) at (.24375,{.1625*\eFiveH});
      \coordinate (KB) at (.59375,{.1625*\eFiveH});
      \coordinate (KC) at (.41875,{.5125*\eFiveH});
      \fill[e5Cap,opacity=.45] (KA)--(KB)--(KC)--cycle;
      \draw[e5Invalid,line width=1.2pt] (KA)--(KB)--(KC)--cycle;
      \coordinate (W) at (.309375,{.20625*\eFiveH});
      \draw[e5Gray!50,line width=.55pt,dash pattern=on 2pt off 2pt] (A)--(M);
      \node[e5 point,fill=e5Valid,draw=white,line width=.4pt,minimum size=4.4pt] at (W) {};
      \node[e5 point,fill=e5Invalid] at (M) {};
      \node[small label,above right=3pt,text=e5Invalid] at (M) {$\mu$};
    \fi

    \node[vertex,below=4pt] at (A) {$\mu_a$};
    \node[vertex,below=4pt] at (B) {$\mu_b$};
    \node[vertex,above=4pt] at (C) {$\mu_c$};
  \end{scope}
}
\begin{scope}[shift={(12.6,0)},x=3.5cm,y=3.5cm]
  \coordinate (B) at (1,0);
  \coordinate (N) at (0,0);
  \coordinate (C) at (.5,\eFiveH);
  \coordinate (CB) at (.6,{.8*\eFiveH});
  \coordinate (CD) at (.4,{.8*\eFiveH});
  \coordinate (BC) at (.925,{.15*\eFiveH});
  \coordinate (BD) at (.85,0);
  \draw[e5Gray!65,line width=1.1pt] (B)--(N)--(C)--cycle;
  \filldraw[fill=e5Cap,draw=e5Invalid,line width=1.2pt]
    (CB)--(CD)--(C)--cycle;
  \filldraw[fill=e5Cap,draw=e5Invalid,line width=1.2pt]
    (B)--(BD)--(BC)--cycle;
  \draw[e5Invalid,line width=1.5pt] (B)--(C);
  \foreach \n/\tint in {2/50,4/75,8/100}{
    \draw[e5Valid!\tint,line width=1pt,dash pattern=on 2.4pt off 2pt]
      ({1-1/\n},0)--({.5*(1-1/(2*\n-1))},{\eFiveH*(1-1/(2*\n-1))});
  }
  \draw[->,e5Valid,line width=.9pt] (.43,.27)--(.75,.36);
  \node[small label,text=e5Valid] at (.265,.11) {$a\to\infty$};
  \node[e5 point,fill=e5Valid,draw=white,line width=.4pt,minimum size=4.4pt]
    at (.875,0) {};
  \node[vertex,below=4pt] at (B) {$\mu_b$};
  \node[vertex,below=4pt] at (N) {$\mu_a$};
  \node[vertex,above=4pt] at (C) {$\mu_c$};
\end{scope}
\end{tikzpicture}
\endgroup}
				\caption{E1 implies validity and invalidity are regions on the probability simplex $\Delta$.
					(a) \ref{axiom:evalue_mixing} makes both regions convex.
					(b) \ref{axiom:stability_of_invalidity} requires an invalid copy $(1-q)\mu + q\Delta$ of $\Delta$ around each invalid $\mu$, for some $q\in(0,1)$.
					(c) This prevents the boundary from being invalid.
					(d) The notion of validity $\Ex[\delta] \geq 1$ restricted to the face $\{a,b=0,c=1/2\}$ of the simplex.
						The validity boundary approaches the invalid laws $\mu_b$ and $\mu_c$ as $a$ increases, violating \ref{axiom:stability_of_invalidity}.}
				\label{fig:arbitrary_calibration_geometry}
			\end{figure}	
						
		\subsection{Decision making with e-values}\label{sec:optimal_decisions}
			An immediate consequence of Theorem~\ref{thm:expected_loss} is that, under the axioms, $\delta$ is valid if and only if $e = L_P(\delta)$ is a valid e-variable.
			
			To appreciate what this means, it helps to again view e-values as numerical representations of underlying decisions.
			In this general decision setting, the e-value realization $e(x) = t$ identifies the equivalence class of decisions
			\begin{align*}
				D_{t}
					= L_P^{-1}(\{t\})
					= \{d \in \D : L_P(d) = t\}.
			\end{align*}
			All decisions in such a class are equivalent in terms of their contribution to validity.
			This means that the e-variable $e = L_P(\delta)$ can be viewed as being distributed over labels of these equivalence classes.
			
			A consequence is that finding a good e-variable is equivalent to finding a good decision variable up to the class.
			For example, if the class $D_t$ has utility $U(t)$, a valid decision variable $\delta$ with label $e=L_P(\delta)$ has expected utility $\Ex^Q[U(e)]$ under a distribution $Q$.
			Any measurable decision rule selecting from $D_{e(x)}$ at each $x$ has the same loss $e$, and so is valid and receives the same utility $U(e)$.
			
			In fact, Theorem~\ref{thm:optimal} shows that finding an optimal decision variable can be reduced to finding an optimal e-variable.
			In particular, a measurable utility function $u : \D \to [-\infty, \infty]$ on decisions induces a utility function over e-values:
			\begin{align*}
				U(t) 
					= \sup_{d \in \bigcup_{s \leq t} D_s} u(d),
			\end{align*}
			where $d$ ranges over $\bigcup_{s \leq t} D_s$, as decisions with lower losses are also valid.
			To find a valid and optimal decision variable $\delta^*$, it therefore suffices to find a valid e-variable $e^*$ maximizing $\Ex^Q[U(e)]$ and then select the best decision from $\bigcup_{s \leq e^*(x)} D_s$.
			Such utility-optimal e-variables are studied in \citet{koning2024continuous}.
			
			\begin{theorem}[Optimal decision variable from optimal e-variable]\label{thm:optimal}
				Let $e^*$ maximize $\Ex^Q[U(e)]$ over valid e-variables.
				Let $\delta^*$ be a decision variable such that $\delta^*(x)$ maximizes $u(d)$ subject to $L_P(d) \leq e^*(x)$ at each $x$.
				Then $\delta^*$ is valid and
				\begin{align*}
					\Ex^Q[u(\delta^*)]
						= \sup_{\delta : \textnormal{ valid}}\Ex^Q[u(\delta)]
						= \Ex^Q[U(e^*)].
				\end{align*}
			\end{theorem}
			\begin{proof}
				Since $L_P(\delta^*) \leq e^*$, the decision variable $\delta^*$ is valid.
				Every other valid decision variable $\delta$ supplies a competing e-variable $L_P(\delta)$, so
				\begin{align*}
					\Ex^Q[u(\delta)]
						\leq \Ex^Q[U(L_P(\delta))]
						\leq \Ex^Q[U(e^*)]
						= \Ex^Q[u(\delta^*)],
				\end{align*}
				where the final equality holds because $\delta^*$ selects a best decision.
			\end{proof}
			
			\begin{remark}[Optimal decisions under composite validity]\label{rmk:optimal_composite}
				The reduction from optimal decisions to optimal e-variables extends to the composite setting, where validity is required for every $P \in \P$.

				In the composite setting, a decision variable $\delta$ produces an e-variable $e_P = L_P(\delta)$ for every $P\in \P$.
				An attained collection of e-values $t = (t_P)_{P \in \P}$ now identifies the equivalence class of decisions
				\begin{align*}
					D_t = \{d \in \D : L_P(d) = t_P \text{ for every } P \in \P\}.
				\end{align*}
				
				The utility of these loss bounds is again $U(t) = \sup_{d \in \bigcup_{s \leq t} D_s} u(d)$, where $s \leq t$ means $s_P \leq t_P$ for every $P \in \P$.
				
				As in Theorem~\ref{thm:optimal}, it suffices to maximize $\Ex^Q[U((e_P)_{P \in \P})]$ over families of e-variables $(e_P)_{P \in \P}$, with each $e_P$ valid under $P$.
				Given a maximizing family $(e_P^*)_{P\in\P}$, any decision variable that maximizes $u(d)$ at each $x$ subject to $L_P(d) \leq e_P^*(x)$ for every $P \in \P$ is optimal under composite validity.
				
				If $\P$ is finite and $U$ is additively separable, $U(t) = \sum_{P\in\P}U_P(t_P)$, the expected utilities $\Ex^Q[U_P(e_P)]$ may be maximized separately.
				If the loss $L_P = L$ is common across $P \in \P$, the optimization reduces to finding a single e-variable valid for $\P$.
			\end{remark}
		
		\subsection{Recovering calibration to classical testing}\label{sec:recovering_classical_calibration}
			We began this section by replacing calibration to classical testing.
			We now show in what way the calibration to classical testing can be recovered.
						
			The key idea is to label an invalid decision $d$ by the supremal probability with which it can appear in a valid law:
			\begin{align}\label{eq:induced_level}
				\alpha(d)
					:= \sup\{q \in [0, 1] : q \mu_d + (1 - q) \nu \textnormal{ is valid for some law } \nu\}
					\equiv q_{d}.
			\end{align}
			Here, Axiom \ref{axiom:stability_of_invalidity} ensures $\alpha(d) < 1$.
			
			Corollary~\ref{cor:recovered_calibration} shows how calibration to classical testing emerges, if $\D$ has a least decision $n$: $n \precsim d$, for every $d \in \D$.
			In particular, if we label the equivalence class $D_{L(d)}$ by $r_{\alpha(d)}$, we recover Axiom~\ref{axiom:evalue_calibration} for $\{n,d\}$-valued decision variables.
			
			\begin{corollary}\label{cor:recovered_calibration}
				Assume Axioms \ref{axiom:evalue_distribution}, \ref{axiom:evalue_mixing}, \ref{axiom:stability_of_invalidity} hold and $0 < q_{\delta_0} < 1$ for some $\delta_0$.
				Suppose the normalized loss $L$ from Theorem~\ref{thm:expected_loss} attains its minimum at $n$.
				Then, for every decision $d \in \D$ with $0 < \alpha(d) < 1$, every $\{n, d\}$-valued decision variable $\delta$ satisfies
				\begin{align*}
					\delta \textnormal{ is valid} \iff P(\delta = d) \leq \alpha(d).
				\end{align*}
			\end{corollary}
			\begin{proof}
				Since $L(n) = \inf L = 0$, mixing $d$ with $n$ yields the smallest possible expected loss $qL(d)$ for each $q$.
				Hence, the largest mixing probability in \eqref{eq:induced_level} is
				\begin{align}\label{eq:alpha_L}
					\alpha(d) = 1 / L(d),	
				\end{align}
				attained by mixing with $\mu_n$.
				For a $\{n, d\}$-valued decision variable $\delta$, Theorem~\ref{thm:expected_loss} gives validity exactly when $\Ex^P[L(\delta)] = P(\delta = d) / \alpha(d) \leq 1$.
			\end{proof}

			\begin{remark}[Extending to $\alpha > 1$]
				Classical testing does not label decisions corresponding to e-values $0 < t \leq 1$ and $t = \infty$.
				To gain an intuition for such decisions from the perspective of classical testing, an idea is to extend the concept of a significance level by taking \eqref{eq:alpha_L} as its definition: $\alpha(d) := 1 / L(d)$.
				This results in an evidence space $\{r_\alpha, \alpha \in [0, \infty]\}$, with a non-rejection $n$ corresponding to $r_\infty$.
				
				While this loses the interpretation as a probability, we can obtain a feeling for a decision $d$ with $\alpha(d) > 1$ by comparing it to a conservative test.
				In particular, for $q \in (0, 1]$ and $\alpha \in (0, 1)$ we have that a constant decision variable $r_{\alpha / q}$ has the same expected loss $q/\alpha$ as a decision variable with law $(1 - q) \mu_n + q \mu_{r_\alpha}$.
				
				For example, this means that reporting a rejection at `level $\alpha = 2$' ($r_2$) with certainty has the same expected loss as a level $0.05$ test that rejects with (conservative) probability $0.025$.
				This quantifies in what sense such a test is `more valid' than a level $0.05$ test that rejects with probability $0.05$.
			\end{remark}

	\section{Supported decisions and closure}
		Section \ref{sec:optimal_decisions} shows how e-values emerge in general decision problems: as labels of validity-based equivalence classes of decisions.
		We now generalize this to arbitrary evidence variables.
		This extends the ideas pioneered by \citet{grunwald2023posterior, grunwald2024beyond} beyond e-values and expected-losses: see Section \ref{sec:grunwald}.
		It also unifies several recent closure results \citep{koning2023post, xu2025bringing, koning2026measure}.
		
		\subsection{Linking valid evidence to valid decisions}
			Let $\D$ and $\E$ be measurable spaces, where we view $\D$ as a decision space and $\E$ as an evidence space.
			We equip $\E$ with a preorder $\precsim$, and we assume that validity of evidence variables on $\E$ is monotone with respect to this preorder (Axiom \ref{axiom:monotonicity}).
			Such a preorder always exists: equality `$=$' of decisions is a preorder.
			
			We link the two spaces by a measurable map $\pi : \D \to \E$ that represents the validity:
			\begin{align*}
				\delta \textnormal{ is valid} \iff \pi(\delta) \textnormal{ is valid}.
			\end{align*}
			This abstracts Theorem~\ref{thm:expected_loss}, where such a map appears as the loss function $\pi = L_P$, with $(\E, \precsim) = ([0, \infty], \leq)$ and the composite setting in Remark \ref{rmk:optimal_composite} with $\pi = (L_P)_{P \in \P}$ with $\E = [0, \infty]^\P$ and $\precsim$ element-wise inequality.
			
		\subsection{Post-hoc valid decision sets}
			A \emph{decision set} is a measurable set-valued map $C : \X \to 2^\D$.
			The idea is that a decision set presents a collection of valid decisions that we can freely choose from \emph{post-hoc}.
			
			\begin{definition}[Post-hoc valid]
				A decision set $C$ is post-hoc valid if every measurable selection $\delta(x) \in C(x)$ is valid.
			\end{definition}
			
			A post-hoc valid decision set naturally arises from an evidence variable, by transferring its validity to decisions.
			In particular, we say that a decision $d$ is \emph{supported} by the evidence $\e(x)$ if $\pi(d) \precsim \e(x)$, and denote the set of all supported decisions by
			\begin{align*}
				\S_\e(x)
					:= \{d \in \D : \pi(d) \precsim \e(x)\}.
			\end{align*}
				
			\begin{proposition}\label{prp:supported_decisions}
				If $\e$ is valid, then $\S_\e$ is post-hoc valid.
			\end{proposition}
			\begin{proof}
				Every measurable selection $\delta$ from $\S_\e$ satisfies $\pi(\delta) \precsim \e$, so monotonicity gives validity of $\pi(\delta)$ whenever $\e$ is valid.
			\end{proof}

			We can also recover a valid evidence variable from a post-hoc valid decision set.
			For a decision set, we define its least supporting evidence as
			\begin{align*}
				\e_C(x)
					:= \sup_{\precsim} \{\pi(d) : d \in C(x)\}.
			\end{align*}
			Under continuity from below and the selection assumption in Proposition~\ref{prp:least_supporting_evidence}, the least supporting evidence is valid if and only if the decision set is post-hoc valid.
			For e-validity, continuity from below holds by the monotone convergence theorem.
	
			\begin{evalueaxiom}[Continuity from below]\label{axiom:continuity_below}
				If every evidence variable $\e_k$ is valid and $\e_k \uparrow \e$ pointwise then $\e$ is valid.
			\end{evalueaxiom}
		
			\begin{proposition}\label{prp:least_supporting_evidence}
				Assume \ref{axiom:monotonicity}, \ref{axiom:continuity_below}.
				Suppose $\e_C$ exists measurably and $C$ admits measurable selections $\delta_k$ such that $\pi(\delta_k) \uparrow \e_C$ pointwise.
				Then $C$ is post-hoc valid if and only if $\e_C$ is valid.
			\end{proposition}
			\begin{proof}
				If $\e_C$ is valid every measurable selection from $C$ satisfies $\pi(\delta) \precsim \e_C$.
				Monotonicity then makes $\pi(\delta)$ valid, so that $\delta$ is valid.
				Conversely, if $C$ is post-hoc valid, each $\delta_k$ is valid, and hence so is each $\pi(\delta_k)$.
				Continuity from below then gives validity of $\e_C$.
			\end{proof}
			
			\begin{remark}
				In the context of Section \ref{sec:optimal_decisions}, under the same selection assumption, $C$ is post-hoc valid if and only if $\Ex^P\left[\sup_{d \in C(X)} L_P(d)\right] \leq 1$.
			\end{remark}
			
			\begin{remark}
				For Proposition \ref{prp:supported_decisions} it suffices that validity of $\pi(\delta)$ implies validity of $\delta$.
				The converse in Proposition \ref{prp:least_supporting_evidence} uses the reverse direction.
			\end{remark}
					
		\subsection{Selecting good and optimal decisions}\label{sec:decision_preferences}
			Given a post-hoc valid decision set $C$, we are free to choose any decision in $C(x)$ after observing $x$.
			To express which decision we prefer, we equip $\D$ with a preorder $\preceq$, where $d \preceq d'$ means that $d'$ is weakly preferred over $d$.
			As with the evidence preorder, we interpret $\delta \preceq \delta'$ pointwise: $\delta(x) \preceq \delta'(x)$ for every $x$.
			
			\begin{example}[Preferences on sets]\label{exm:decision_strength}
				If a decision $d$ represents a prediction or confidence set, we generally prefer smaller sets.
				We can express this by taking $d \preceq d'$ whenever $d' \subseteq d$.
				For sets of discoveries, we generally prefer larger sets, so that we can take $d \preceq d'$ if $d' \supseteq d$.
				We could also compare sets by their measure or cardinality.
			\end{example}
			
			Suppose that we can measurably select a greatest element $\delta_C^*(x)$ from $C(x)$, for every $x$.
			Then $\delta_C^*$ is valid by the post-hoc validity of $C$, and every other measurable selection $\delta$ from $C$ satisfies $\delta \preceq \delta_C^*$.
			In particular, for an evidence variable $\e$, we can consider the best decision variable $\delta_{\e}^*$ from $\S_\e$.
			
			In Proposition \ref{prp:optimal_decision_from_optimal_evidence}, we present an abstraction of Theorem \ref{thm:optimal}, showing how an optimal evidence variable can lead to an optimal decision variable.
			To compare decision variables, let $\unlhd$ be a preorder on decision variables that is compatible with the pointwise preorder $\preceq$: $\delta \preceq \delta' \implies \delta \unlhd \delta'$.
			In Theorem \ref{thm:optimal}, the preorder is $\Ex^Q[u(\delta)] \leq \Ex^Q[u(\delta')]$, which is compatible with the pointwise order $u(\delta) \leq u(\delta')$.
			
			The preorder on decision variables induces a preorder on evidence variables through the decisions they support.
			We write $\e \unlhd \e'$ if, for every decision variable $\delta$ supported by $\e$, there is a decision variable $\delta'$ supported by $\e'$ with $\delta \unlhd \delta'$.
			If $\delta_{\e}^*$ and $\delta_{\e'}^*$ both exist, this simply means $\delta_{\e}^* \unlhd \delta_{\e'}^*$.
			This is analogous to how the utility $U$ on e-values is induced by a utility $u$ on decisions in Section \ref{sec:optimal_decisions}.
			
			\begin{proposition}\label{prp:optimal_decision_from_optimal_evidence}
				Suppose a valid evidence variable $\e^*$ is preferred over every other valid evidence variable $\e$: $\e^* \unrhd \e$.
				If $\delta_{\e^*}^*$ exists, then it is valid and preferred over every other valid decision variable $\delta$: $\delta_{\e^*}^* \unrhd \delta$.
			\end{proposition}
			\begin{proof}
				Validity of $\delta_{\e^*}^*$ follows from post-hoc validity of $\S_{\e^*}$.
				Now take any valid $\delta$.
				Its own evidence variable $\pi(\delta)$ is valid and supports $\delta$.
				Since $\e^* \unrhd \pi(\delta)$, by definition $\e^*$ supports some $\delta'$ with $\delta \unlhd \delta'$.
				Since $\delta_{\e^*}^*$ selects a greatest element of $\S_{\e^*}$ pointwise, $\delta' \preceq \delta_{\e^*}^*$, and hence $\delta \unlhd \delta' \unlhd \delta_{\e^*}^*$.
			\end{proof}
			
		\subsection{Improving decision variables and closure}\label{sec:closure}
			We can now use the tools we developed to unpack the mathematical structure underlying recent closure results.
			The underlying idea is to start with some valid decision variable and improve it.
			
			A decision variable $\delta$ can sometimes be improved by decisions supported by its own evidence:
			\begin{align*}
				\delta(x) \preceq \delta'(x) \in \S_{\pi(\delta)}(x), \textnormal{ for every } x.
			\end{align*}
			
			In particular, whenever a greatest element can be selected measurably at each $x$, we can consider the best supported improvement of a decision variable $\delta$:
			\begin{align*}
				\overline{\delta}
					:= \max_{\preceq} \S_{\pi(\delta)}.
			\end{align*}
			The following result shows that the map $\delta \mapsto \overline{\delta}$ is a closure operator, under a compatibility assumption between $\preceq$ and $\precsim$.
			Here, $\simeq$ denotes indifference under $\preceq$.
			
			\begin{proposition}[Closure operator]\label{prp:closure}
				The map $\delta \mapsto \overline{\delta}$ preserves validity, is extensive $\delta \preceq \overline{\delta}$ and idempotent $\overline{\overline{\delta}} \simeq \overline{\delta}$.
				If $d \preceq d'$ $\implies$ $\pi(d) \precsim \pi(d')$ then it is also monotone in $\preceq$ and satisfies $\pi(\overline{\delta}) \sim \pi(\delta)$, and hence induces a closure operator on $\simeq$-equivalence classes.
			\end{proposition}
			\begin{proof}
				Since $\delta(x) \in \S_{\pi(\delta)}(x)$, we have $\delta \preceq \overline{\delta}$.
				As $\overline{\delta}$ is a measurable selection from $\S_{\pi(\delta)}$, it satisfies $\pi(\overline{\delta}) \precsim \pi(\delta)$ and is valid whenever $\delta$ is valid.
				The inclusion $\S_{\pi(\overline{\delta})} \subseteq \S_{\pi(\delta)}$ gives $\overline{\overline{\delta}} \preceq \overline{\delta}$, while extensivity gives the reverse.
				Under the additional condition, extensivity also gives $\pi(\delta) \precsim \pi(\overline{\delta})$, and so $\pi(\overline{\delta}) \sim \pi(\delta)$.
				Finally, $\delta' \preceq \delta$ implies $\S_{\pi(\delta')} \subseteq \S_{\pi(\delta)}$, so that their greatest members are ordered.
			\end{proof}
			
			\begin{remark}[Recertification]
				Even when a valid decision variable $\delta$ is closed, its evidence $\pi(\delta)$ may still be loose in the sense that it can be strictly increased without losing validity.
				Increasing $\pi(\delta)$ in this way can only enlarge the set of supported decisions, and may allow us to select a decision variable strictly preferred to $\delta$.
				In parallel work, \citet{bai2026uniform} develop this idea for e-variables under the name `recertification'.
			\end{remark}
			
		\subsection{Examples}
			\begin{example}[Decisions supported by p-variables]
				Suppose we call a decision variable $\delta$ valid if its loss satisfies the tail bound
				\begin{align*}
					P(L_P(\delta(X)) \geq 1/\alpha) \leq \alpha,\quad \textnormal{for every } \alpha \in (0, 1).
				\end{align*}
				This is equivalent to requiring $\pi(\delta)$ to be a valid p-variable, where $\pi(d) = 1/(L_P(d) \vee 1)$.
				Since smaller p-variables represent stronger evidence, a valid p-variable $p$ supports the decision set
				\begin{align*}
					\S_p(x)
						= \{d : p(x) \leq \pi(d)\}
						= \{d : L_P(d) \leq 1/p(x)\}.
				\end{align*}
				Conversely, the least supporting evidence for a decision set $C$ is
				\begin{align*}
					p_C(x)
						= \inf_{d \in C(x)} \pi(d)
						= \left(\sup_{d \in C(x)} L_P(d) \vee 1\right)^{-1}.
				\end{align*}
				Since a decreasing limit of valid p-variables is valid, Proposition \ref{prp:least_supporting_evidence} gives that $C$ is post-hoc valid exactly when $p_C$ is a valid p-variable, under the selection assumption.
			\end{example}
			
			\begin{example}[Predictive loss with a single e-variable]
				Suppose we observe $X$, choose a decision $d$, and then incur a non-negative loss $L(Y, d)$ when the outcome $Y$ is revealed.
				Let $e(X, Y)$ be an e-variable that is valid under $\P$.
				
				After observing $X = x$, the available evidence is given by the e-prediction $y \mapsto e(x, y)$ \citep{koning2025fuzzy, vovk2025conformal}.
				Representing each decision by its loss function $\pi(d) = L(\cdot, d)$ and pointwise order, we obtain the supported decisions
				\begin{align*}
					\S_e(x)
						= \{d \in \D : L(y, d) \leq e(x, y), \textnormal{ for every } y	\}.
				\end{align*}
				Every measurable selection $\delta(x) \in \S_e(x)$ is valid:
				\begin{align*}
					\Ex^P[L(Y, \delta(X))]
						\leq \Ex^P[e(X, Y)]
						\leq 1, \quad \textnormal{ for every } P \in \P,
				\end{align*}
				so that $\S_e$ is a post-hoc valid decision set under this notion of validity.
			\end{example}

	\section{Applications}
		\subsection{A foundation for `e-closure' in multiple testing}
			\citet{xu2025bringing} study expectation-based multiple testing guarantees, including FDR and FWER control.
			We show how our results axiomatically justify such expected-loss guarantees, and show how their `e-closure principle' arises from supporting decisions with evidence.
			
			Let $\P$ be a collection of distributions and let $\H$ denote a finite collection of hypotheses.
			Our decisions are rejection sets $R \subseteq \H$.
			The following corollary to Theorem \ref{thm:expected_loss} characterizes expectation-based multiple testing guarantees.
			For example, FDR control at level $\alpha \in (0, 1)$ corresponds to
			\begin{align*}
				L_P(R) = \textnormal{FDP}_P(R) / \alpha.
			\end{align*}

			\begin{corollary}
				Axioms \ref{axiom:evalue_distribution}, \ref{axiom:evalue_mixing} and \ref{axiom:stability_of_invalidity} hold for every $P \in \P$ if and only if there exist measurable losses $L_P$ such that, for every $P$ and every decision variable $R(X)$,
				\begin{align*}
					R(X) \textnormal{ is valid under }P \iff \Ex^P[L_P(R(X))] \leq 1.
				\end{align*}
			\end{corollary}
			
			Let $(e_P)$ be a family of e-variables, each valid for $P \in \P$.
			The rejection sets supported by this family are
			\begin{align*}
				\S_{(e_P)}(x)
					= \{R \subseteq \H : L_P(R) \leq e_P(x), \textnormal{ for every } P \in \P\}.
			\end{align*}
			For FDR, the supporting inequality is simply $\textnormal{FDP}_P(R)/\alpha \leq e_P(x)$.
			
			The following consequence of Proposition \ref{prp:least_supporting_evidence} recovers the main result of \citet{xu2025bringing}.
			
			\begin{corollary}\label{cor:e-closure}
				A non-empty collection of rejection sets $C$ is post-hoc valid under every $P \in \P$ if and only if it is supported by a family $(e_P)$ of e-variables, with each $e_P$ valid under $P$.
			\end{corollary}
			\begin{proof}
				Support implies post-hoc validity by Proposition~\ref{prp:supported_decisions}.
				Conversely, for each $P$, choose a measurable selection from $C$ with maximal loss $e_P(x):=\max_{R\in C(x)}L_P(R)$.
				Post-hoc validity makes this selection valid, so each e-variable $e_P$ is valid under $P$ and the family $(e_P)$ supports $C$.
			\end{proof}
			
			\begin{example}[FWER and set inclusion closure]
				For FWER at level $\alpha$, the loss $L_P(R)=\ind{\exists H\in R:\ P\in H}/\alpha$ is nondecreasing under set inclusion-preference, and unions of supported rejection sets remain supported.
				Hence Proposition \ref{prp:closure} gives the closure of a valid rejection rule $R$ directly as
				\begin{align*}
					\overline R(x)
						:= \max_{\subseteq}\{R' \subseteq \H : L_P(R') \leq L_P(R(x)),\ \textnormal{for every }P\in\P\}.
				\end{align*}
				Thus every FWER-valid rejection rule is dominated by a closed FWER-valid rule, as in the classical closure principle \citep{marcus1976closed, xu2025bringing}.
			\end{example}

			\begin{example}[FDR and cardinality improvements]
				Unlike FWER, the rejection sets supported under FDR validity need not have a greatest element under set inclusion; see Example 23 in \citet{xu2025bringing}.
				The closed eBH procedure of \citet{xu2025bringing} and the closed BH procedure of \citet{goeman2026uniform} each report a supported set of maximum \emph{cardinality}.
				In our framework, this is the best supported decision $\max_{\preceq}\S_{(e_P)}$ from Section~\ref{sec:decision_preferences}, with preference for greater cardinality: $R \preceq R' \iff |R| \leq |R'|$.
				By Corollary~\ref{cor:e-closure}, every selection from $\S_{(e_P)}$ controls FDR.
			\end{example}
			
		\subsection{Choosing the decision problem post-hoc}\label{sec:grunwald}
			\citet{grunwald2024beyond} considers a setting in which the problem itself can be selected after observing the data.
			We can recover this in our framework by taking decisions of the form $d = (b, a)$, where $b$ identifies a problem and $a \in \A_b$ the action taken.
			Under the axioms of Theorem \ref{thm:expected_loss}, validity under each $P \in \P$ is represented by a loss $L_P$ on these pairs.
			Write $\ell_{b, P}(a) := L_P((b, a))$ for the loss of action $a$ in problem $b$.
			
			Let $(e_P)$ be a family of valid e-variables, with $e_P$ valid under $P \in \P$.
			The decision set supported by this family is
			\begin{align*}
				\S_{(e_P)}(x)
					= \{(b, a) : a \in \A_b, \ell_{b, P}(a) \leq e_P(x), \textnormal{ for every } P \in \P\}.
			\end{align*}
			This set is post-hoc valid under every $P \in \P$, so both the problem and action can be chosen post-hoc.			
			This recovers Gr\"unwald's guarantee for decision rules compatible with a family of valid e-variables.
			
			\begin{remark}
				This construction highlights how general a `decision' can be: even the problem itself can be part of the decision.
			\end{remark}
			
		\subsection{E-measures}
			\citet{koning2026measure} studies jointly presenting evidence against an entire class of hypotheses $\H \subseteq 2^\P$.
			Where a single hypothesis $H$ represents a binary claim about the true distribution $P^*$, such a collection of binary claims represents a quantitative claim about $P^*$.
			
			To cast this into our framework, we take a decision $d$ to be an entire `e-function' $d = (H \mapsto e(H))$, and we write the associated decision variable $\delta(X)$ as $H \mapsto e(H \mid X)$.
			We choose the loss\footnote{We include $\emptyset$ in $\H$ and use the conventions $0\cdot\infty=0$, $\sup\emptyset=0$ and $\inf\emptyset=\infty$. We assume the displayed losses and pointwise closures are measurable.}
			\begin{align*}
				L_P(e)
					:= \sup_{H \in \H} e(H)1_H(P)
					= \sup_{H \ni P} e(H),
			\end{align*}
			which measures the strongest evidence against a true hypothesis under $P$.
			Familywise or `uniform' validity of the entire e-function is then
			\begin{align*}
				\Ex^P\left[\sup_{H \ni P} e(H\mid X)\right] \leq 1, \quad \textnormal{ for every } P \in \P,
			\end{align*}
			matching the expected-loss framework recovered from Axioms \ref{axiom:evalue_distribution}, \ref{axiom:evalue_mixing}, \ref{axiom:stability_of_invalidity}.
			For $e(H)\in\{0, 1/\alpha\}$ on every nonempty $H\in\H$, this is familywise error rate control at level $\alpha$.
			
			Using pointwise dominance as the preference order $\preceq$, the closure of $e$ is the greatest e-function that preserves the evidence $\sup_{H \ni P} e(H)$ for every $P$.
			The following corollary to Proposition \ref{prp:closure} shows that it must be of the `minitive-measure' form
			\begin{align*}
				\overline{e}\left(\bigcup_{H \in \S} H\right)
					= \inf_{H \in \S} \overline{e}(H),
			\end{align*}
			whenever the union belongs to $\H$, which \citet{koning2026measure} calls an `e-measure'.
			
			\begin{corollary}[Closure into e-measure]\label{cor:e_measure_closure}
				The best e-function supported by $e$ under every $P$ is
				\begin{align*}
					\overline e(H)
						:= \inf_{P \in H} L_P(e)
						=  \inf_{P \in H} \sup_{G \ni P} e(G).
				\end{align*}
				The closed e-function is an e-measure.
				The map $e\mapsto\overline e$ is a closure operator satisfying $e \leq \overline{e}$ and  $L_P(\overline e)=L_P(e)$.
			\end{corollary}
			\begin{proof}
				An e-function $e'$ is supported by $e$ under every $P$ precisely when $L_P(e')\leq L_P(e)$ for every $P$, or equivalently, $e'(H)\leq\inf_{P\in H}L_P(e)=\overline e(H)$ for every $H$.
				Thus $\overline e$ is the greatest supported e-function.
				Since pointwise dominance implies loss dominance, Proposition \ref{prp:closure} gives the closure operator, domination $e\leq\overline e$, and validity equivalence $L_P(\overline e)=L_P(e)$ for every $P$.
				Applied pointwise, it preserves uniform validity.
				Finally, taking infima over unions gives the minitivity identity, so $\overline e$ is an e-measure.
			\end{proof}
			
			For antitonic e-functions with smallest true hypotheses $H_P$, hypothesis-wise validity $\Ex^P[e(H_P\mid X)]\leq1$ for every $P$ already ensures uniform validity of the entire e-function and its closure.
			Under these conditions, the closure becomes $\overline e(H)=\inf_{P\in H}e(H_P)$ \citep{koning2026measure}.		
			
			\begin{proposition}
				Suppose $e$ is antitonic: $e(H \mid x) \geq e(H' \mid x)$ for $H \subseteq H'$.
				Assume every $P$ has a smallest true hypothesis $H_P \in \H$: $P \in H_P$ and $H_P \subseteq H$, for every $H \ni P$.
				
				Then $H \mapsto e(H\mid X)$ is uniformly valid if and only if $\Ex^P[e(H_P\mid X)]\leq1$ for every $P$.
			\end{proposition}
			\begin{proof}
				Antitonicity gives $L_P(e)=\sup_{H\ni P}e(H)=e(H_P)$, since $H_P\subseteq H$ whenever $P\in H$.
			\end{proof}
		
		
		\subsection{Closed post-hoc valid test families}
			A test family $\phi$ is \emph{post-hoc valid} if $\phi(\widetilde{\alpha})$ is a valid evidence variable for \emph{every} measurable choice of level $\widetilde{\alpha}$.
			To cast this into our framework, we take a decision to be an entire test family $d = (\alpha \mapsto \phi(\alpha))$.

			Post-hoc validity of a test family depends only on its p-value.
			Indeed, after observing $x$, we may choose any level at which $\phi$ rejects.
			Since smaller levels give stronger rejections, what matters is how small these levels can be.
			Their infimum is the family's p-value,
			\begin{align*}
				p(x) := \inf\{\alpha : \phi(\alpha)(x)=r_\alpha\}.
			\end{align*}
			The least evidence that supports all these rejections is therefore $r_{p(x)}$, giving us the map $\pi(\phi)(x)=r_{p(x)}$.
			By Proposition~\ref{prp:least_supporting_evidence}, $\phi$ is post-hoc valid if and only if $r_{p(X)}$ is valid.
			Under the criterion of Section~\ref{sec:characterization}, this is precisely the requirement for $p$ to be a valid post-hoc p-variable \citep{koning2023post}.

			A natural preference $\preceq$ over test families is to prefer rejections over non-rejections at each level.
			We can use this to close $\phi$ by letting it reject at every level its own evidence supports.
			At $x$, these are precisely the levels $\alpha\geq p(x)$, including any that $\phi$ itself skips.
			The closure operator of Proposition~\ref{prp:closure} therefore takes the form
			\begin{align*}
				\overline{\phi}(\alpha)(x) =
					\begin{cases}
						r_\alpha, & p(x)\leq\alpha, \\
						n, & p(x)>\alpha.
					\end{cases}
			\end{align*}
			The two families have the same p-value, so $\overline{\phi}$ is post-hoc valid whenever $\phi$ is.
			This means every post-hoc valid test family is dominated pointwise by a post-hoc valid threshold family, recovering the closure result of \citet[Remark~4]{koning2023post}.

	\section{Distinguishing non-rejections}\label{sec:distinguishing_non_rejections}
		Throughout, we have defined tests across different levels using a universal non-rejection decision $n$.
		Using a universal non-rejection decision is a modeling choice: together with law invariance, it means that the level reported on a non-rejection cannot affect validity.
		In this section, we distinguish non-rejections at different levels, defining a level $\alpha$ test to take values in $\{n_\alpha, r_\alpha\}$, where $n_\alpha$ represents a non-rejection at level $\alpha$.
		By distinguishing non-rejections, we can study what happens when the level of the non-rejection decision is allowed to influence validity.

		In this section, we consider the evidence space $\D = \{n_\alpha, r_\alpha: \alpha \in(0,1)\}$.\footnote{We enlarge the sigma-algebra on the evidence space by the sets $\{n_\alpha : \alpha \in B\}$ for Borel $B \subseteq (0, 1)$, and assume that $\alpha \mapsto n_\alpha$ and its inverse are measurable. The enlarged evidence space is therefore again standard Borel.}
		We assume that this evidence space remains totally ordered and that all the decisions $n_\alpha$ and $r_\alpha$ are distinct.
		In particular, if $\beta \leq \alpha$, then $n_\beta \succeq n_\alpha$ and $r_\beta \succeq r_\alpha$, while $n_\beta \prec r_\alpha$ for every $\alpha, \beta \in (0, 1)$.
		This means every rejection is stronger than every non-rejection.
		Moreover, a non-rejection at a smaller level is considered to be a stronger decision.
		We believe this ordering makes sense: failing to reject at level $\alpha = 90\%$ suggests there was no evidence to be found, while failing to reject at level $\alpha = 1\%$ leaves open the possibility that we set the bar too high.

		\subsection{Distinct non-rejections and compensation}
			In this more general setup, Theorem~\ref{thm:expected_loss} still characterizes validity by an expected-loss bound.
			However, this does not pin down a particular notion of validity (choice of loss function).
			Moreover, even with classical calibration, the reported level $\alpha$ of $r_\alpha$ need not equal its recovered level $\alpha(r_\alpha)$ from Section~\ref{sec:recovering_classical_calibration}, since a non-rejection at another level may provide more compensation.
			
			In Proposition~\ref{prp:distinct_non_rejection_losses} we show that if we impose calibration to classical testing (Axiom \ref{axiom:evalue_calibration}) then we still have a range of possible loss values for non-rejections $0 \leq L(n_\alpha) < 1$.
			This means that non-rejections at different levels can differ in their ability to compensate rejection decisions.
			We illustrate this in Example~\ref{exm:arithmetic_validity}.
					
			\begin{proposition}\label{prp:distinct_non_rejection_losses}
				Axioms \ref{axiom:evalue_distribution}, \ref{axiom:evalue_calibration}, \ref{axiom:evalue_mixing} and \ref{axiom:stability_of_invalidity} hold if and only if there exists a measurable loss $L:\D\to[0,\infty)$ such that every evidence variable $\e$ satisfies
				\begin{align*}
					\e\textnormal{ is valid}
						\quad\Longleftrightarrow\quad \Ex^P[L(\e)] \leq 1,
				\end{align*}
				where $\inf_{\alpha\in(0,1)}L(n_\alpha) = 0$ and, for every $\alpha \in (0, 1)$,
				\begin{align}\label{eq:distinct_non_rejection_calibration}
					0\leq L(n_\alpha)<1,
					\qquad (1 - \alpha) L(n_\alpha) + \alpha L(r_\alpha) = 1.
				\end{align}
			\end{proposition}
						
			\begin{example}[Expected level bound]\label{exm:arithmetic_validity}
				Consider the notion of validity
				\begin{align}\label{ineq:arithmetic_mean_level}
					P\left(\phi(\widetilde{\alpha}) = r_{\widetilde{\alpha}}\right)
						\leq \Ex^P[\widetilde{\alpha}],
				\end{align}
				considered in \citet{koning2023post} and subsequently used in \citet{gauthier2025values}, \citet{gauthier2026backward} and \citet{zhu2026beyond}.
				
				This notion of validity satisfies Axioms \ref{axiom:evalue_distribution}, \ref{axiom:evalue_calibration}, \ref{axiom:evalue_mixing} and \ref{axiom:stability_of_invalidity}.
				It can indeed be written in an expected-loss form that distinguishes non-rejections:
				\begin{align*}
					\Ex^P[L(\phi(\widetilde{\alpha}))] \leq 1,
				\end{align*}
				with the loss $L(n_\alpha) = 1 - \alpha$ and $L(r_\alpha) = 2 - \alpha$.

				To see how non-rejections can compensate rejections for this notion of validity, we first consider a $\{n_{0.1}, r_{0.1}\}$-valued evidence variable $\e$ with $P(\e=r_{0.1}) = 0.2$ and $P(\e=n_{0.1})=0.8$.
				This evidence variable is invalid, since
				\begin{align*}
					(2 - 0.1) \times 0.2 + (1 - 0.1) \times 0.8
						= 1.1
						> 1.
				\end{align*}
				Replacing $n_{0.1}$ by $n_{0.3}$ results in a $\{n_{0.3}, r_{0.1}\}$-valued evidence variable that is valid:
				\begin{align*}
					(2 - 0.1) \times 0.2 + (1 - 0.3) \times 0.8
						= 0.94
						< 1.
				\end{align*}
				The rejection decision and its probability remain the same, but the smaller loss of $n_{0.3}$ compensates the rejection loss sufficiently to make the variable valid.
			\end{example}
	
		\subsection{Preserving the spirit of Type-I error control}\label{sec:cross_calibration}
			The question of whether we should distinguish non-rejections is fundamentally a modeling choice about how we should value non-rejections at different levels.
			However, we believe that a universal non-rejection decision more naturally corresponds to the classical interpretations of testing.
			To motivate this, we use an argument based on power.
			
			Suppose that we have some alternative distribution $Q$, and we wish to maximize the probability of obtaining a rejection $r_\alpha$ at some fixed level $\alpha \in (0, 1)$.
			The following result shows that an optimal evidence variable must use the non-rejection with the smallest loss.
			That is, a power optimal evidence variable at level $\alpha$ is generally not $\{n_\alpha, r_\alpha\}$-valued and so not a level $\alpha$ test in the sense of this section, but $\{n, r_\alpha\}$-valued.
			We believe this goes against the spirit of classical testing.
			Example \ref{exm:arithmetic_power} shows what such an optimal test looks like in the setting of Example \ref{exm:arithmetic_validity}.

			\begin{proposition}[Power and non-rejections]\label{prp:power_non_rejections}
				Suppose that Axioms \ref{axiom:evalue_distribution}, \ref{axiom:evalue_calibration}, \ref{axiom:evalue_mixing} and \ref{axiom:stability_of_invalidity} hold, and $L$ is the resulting loss from Proposition \ref{prp:distinct_non_rejection_losses}.
				Assume that $L(n_\beta)\neq L(n_\gamma)$ whenever $\beta\neq\gamma$.
				
				Suppose that $\e$ maximizes $Q(\e = r_\alpha)$ among valid evidence variables, and that $Q(\e = r_\alpha) < 1$.
				Then there exists a unique non-rejection $n\in\{n_\beta:\beta\in(0,1)\}$ with $L(n)=\inf_{\beta\in(0,1)}L(n_\beta) = 0$.
				Moreover, every maximizing evidence variable is $\{n,r_\alpha\}$-valued, $P$- and $Q$-almost surely.
			\end{proposition}

			\begin{example}\label{exm:arithmetic_power}
				For the expected level bound, $L(n_\beta)=1-\beta$ approaches zero as $\beta\uparrow1$.
				To attain this infimum, add $n_1,r_1$ with $L(n_1)=0$ and $L(r_1)=1$.
				A most powerful randomized test of $P$ against $Q$ at level $1/(2-\alpha)$ then attains the optimum by reporting $r_\alpha$ on rejection and $n_1$ otherwise.
				For $\alpha=5\%$, this permits reporting $r_{0.05}$ with probability $1/1.95\approx51.3\%$ under $P$, the supremum on the original decision space.
			\end{example}
		
		\subsection{Upgrading Axiom \ref{axiom:evalue_calibration}}\label{sec:crossed_pair_calibration}
			Based on the discussion in Section \ref{sec:cross_calibration}, we consider an extension of Axiom \ref{axiom:evalue_calibration} to cross-level $\{n_\beta, r_\alpha\}$-valued evidence variables.
			It prevents non-rejections at different levels from influencing validity.

			\begin{extendedcalibrationaxiom}[Calibration of crossed pairs]\label{axiom:cross_calibration}
				Let $\alpha,\beta\in(0,1)$ be fixed.
				A $\{n_\beta,r_\alpha\}$-valued evidence variable $\e$ is valid if and only if $P(\e=r_\alpha)\leq\alpha$.
			\end{extendedcalibrationaxiom}

			Under the assumptions of Proposition~\ref{prp:distinct_non_rejection_losses}, Axiom~\ref{axiom:cross_calibration} gives
			\begin{align*}
				(1-\alpha)L(n_\beta)+\alpha L(r_\alpha)
					= (1-\alpha)L(n_\alpha)+\alpha L(r_\alpha)=1.
			\end{align*}
			Hence $L(n_\beta)=L(n_\alpha)$ for every $\alpha,\beta\in(0,1)$.
			Normalization then gives $L(n_\alpha)=0$, so that $L(r_\alpha)=1/\alpha$.
			This means that while non-rejections can be treated as distinct decisions, they should be valued the same for the sake of validity.

	\section{Numerical average-based characterization}\label{sec:numerical_characterization}
		In this section, we provide a standalone characterization of e-validity through its well-known averaging property.		
		On the numerical evidence space $\D = [0, \infty]$, we show that e-validity is the largest law-invariant notion of validity with this merging property (Axiom \ref{naxiom:valid_averaging}) and benchmark $\e \equiv 1$ (Axiom \ref{naxiom:normalization}).
		
		In light of this result, it is perhaps not surprising that averaging has been shown to be the only admissible form of merging valid e-variables \citep{wang2025admissible, clerico2026geometric}.
		
		\begin{numericalaxiom}[Law Invariance]\label{naxiom:law_invariance}
			If $\e_1$ is valid and $\e_1 \eqd \e_2$, then $\e_2$ is valid.
		\end{numericalaxiom}
		
		\begin{numericalaxiom}[Averaging valid evidence]\label{naxiom:valid_averaging}
			If $(\e_t)_{t\in T}$ is a jointly measurable family of valid evidence variables on a common probability space, then $\e = \int \e_t\,\mu(dt)$ is valid for every probability measure $\mu$ on $T$.
		\end{numericalaxiom}
		
		\begin{numericalaxiom}[Normalization]\label{naxiom:normalization}
			If $\e \equiv c$ is valid, for $c\in[0,\infty]$, then $c \leq 1$.
		\end{numericalaxiom}
		
		\begin{theorem}\label{thm:numerical_maximality}
			The notion of validity $\Ex^P[\e] \leq 1$ is the largest notion of validity that satisfies Axioms \ref{naxiom:law_invariance}-\ref{naxiom:normalization}.
		\end{theorem}
		
		The proof shows that Axioms \ref{naxiom:law_invariance} and \ref{naxiom:valid_averaging} together imply that the constant $\e' \equiv \Ex^P[\e]$ must be a valid evidence variable for every valid $\e$.
		Normalization then forces $\Ex^P[\e] \leq 1$.

		\begin{remark}[Choice of normalization constant]\label{rmk:normalization}
			Note that we can freely change the normalization from $1$ to some other value, if we scale our interpretation accordingly.
			For example, we could replace $1$ by $\alpha$ and restrict the decision space to $\D = [0, 1]$, which recovers classical randomized tests.
			Indeed, continuously interpreted randomized tests at level $\alpha$ are equivalent to e-variables bounded by $1/\alpha$, up to rescaling \citep{koning2024continuous}.
		\end{remark}
			
		\subsection{Relationship between mixing and averaging}\label{sec:mix_vs_average}
			While related, mixing (Axiom \ref{axiom:evalue_mixing_valid}) and averaging (Axiom \ref{naxiom:valid_averaging}) are distinct merging operations.
			Indeed, suppose that $\e_1, \dots, \e_k$ are valid $[0, \infty]$-valued evidence variables, and let $I$ be an independent random index uniform on $\{1, \dots, k\}$.
			Then the validity-preserving part of the independent selection axiom states that the mixture $\e_I$ is valid, whereas the averaging axiom states that its conditional expectation
			\begin{align*}
				\Ex^P[\e_I \mid \e_1, \dots, \e_k]
					= \frac{1}{k} \sum_{i = 1}^k \e_i
			\end{align*}
			is valid.
			Both have the same expectation by the law of iterated expectations: $\Ex^P[\e_I] = \Ex^P[\Ex^P[\e_I \mid \e_1, \dots, \e_k]]$, and are so treated identically by e-validity.

			Averaging is in some sense `better' than mixing as a merging operation: it yields at least as much expected utility of evidence as mixing for every concave utility function $U$.
			Indeed, under any null or alternative distribution $Q$ for which $I$ is independent and uniform, Jensen's inequality and the law of iterated expectations give
			\begin{align*}
				\Ex^Q\left[U\left(\frac{1}{k}\sum_{i=1}^k \e_i\right)\right]
					&\geq \Ex^Q\left[\frac{1}{k}\sum_{i=1}^k U(\e_i)\right] \\
					&= \Ex^Q\left[\Ex^Q\left[U(\e_I) \mid \e_1, \dots, \e_k\right]\right] \\
					&= \Ex^Q[U(\e_I)],
			\end{align*}
			whenever these expectations are well-defined.
			This is a Rao--Blackwell improvement applied to the mixture $\e_I$ \citep{de2025rao}, as previously observed by \citet[Proposition 5.3]{ramdas2025hypothesis}.

	\section{Discussion}
		One of the main takeaways of this work is that the choice between e-values and p-values can be reduced to a choice of principles, when both are viewed as an extension of a test to a rejection at a data-driven level.
		We hope this will facilitate the `e versus p'-debate \citep{chugg2026evidence}.
		
		A second takeaway is that we establish a mathematical bridge between valid decision making and valid evidence.
		The goal of statistics is often to inform downstream decision making, but the translation from the statistical output to decisions is usually opaque.
		For example, one may report a point estimator combined with its standard error, which are then interpreted by a decision maker to inform their decisions.
		The bridge we establish makes this translation concrete, giving exactly the decisions which the evidence supports.
		This can be done even if the statistician has no knowledge of the downstream decision problem, though knowledge of the decision problem can inform the way evidence should be collected.
		Here, e-values emerge under, in our view, reasonable axioms for validity in decision making.
		
		As a final point, we stress that the axiomatization we present here is not unique (see e.g. Section \ref{sec:numerical_characterization}).
		We eventually settled on Axiom \ref{axiom:evalue_mixing} and \ref{axiom:stability_of_invalidity} because the selection of a decision variable applies to any decision variable and has a clear operational meaning.

	\section{Acknowledgements}
		We thank the participants of the e-readers reading group, the attendees of SAVI2026, and Jelle Goeman in particular for feedback.
		The author acknowledges the use of ChatGPT-6 Astra for use as a sparring partner, proof assistance, making the figures and help with editing.

	\bibliographystyle{plainnat}
	\bibliography{bibliography}

	\newpage
	\appendix

	\section{Testing by betting}\label{sec:betting}
		In this section, we show how our work relates to the testing-by-betting framework of \citet{shafer2021testing}, in which a forecaster and skeptic play a betting game. 
		
		We consider a bookmaker (the `forecaster') whose beliefs are represented by the probability distribution $P$.
		A decision $d \in \D$ represents the bookmaker's net loss in euros, so that smaller values of $d$ are preferred by the bookmaker, and larger values by the bettor (the `skeptic').
		A bet is represented by a decision variable $\delta$, which produces a random monetary outcome.
		
		We call a bet `valid' if the bookmaker weakly prefers taking it to declining it, and `invalid' otherwise.
		We can then translate our axioms to this betting setup:
		\begin{itemize}
			\item Axiom \ref{axiom:evalue_distribution} means that the bookmaker's verdict depends only on the distribution of the monetary outcome under $P$.
			\item Axiom \ref{axiom:evalue_mixing} means that if the bookmaker is almost surely willing to take a bet tomorrow (when more information is revealed), then they are also willing to take it today. Similarly, if they would almost surely decline it tomorrow, then they would also decline it today.
			\item For Axiom \ref{axiom:stability_of_invalidity}, suppose the bookmaker declines a bet, such as a sure loss of \EUR{100}. Now construct a new bet, in which that loss is replaced by a large payout with some small probability $q$. Axiom \ref{axiom:stability_of_invalidity} then states that there exists a sufficiently small $q > 0$ below which the bookmaker still declines the new bet, however large the gain.
		\end{itemize}
		
		Under these axioms, Theorem \ref{thm:expected_loss} gives a non-negative measurable loss function $L_P$ such that
		\begin{align*}
			\textnormal{Bookmaker accepts bet } \delta \iff \Ex^P[L_P(\delta)] \leq 1.
		\end{align*}
		Here, $L_P(d)$ is the loss the bookmaker experiences when losing a monetary outcome $d$.
		
		To recover Shafer's betting interpretation of the e-variable, we now suppose that the bookmaker is \emph{risk-neutral}.
		
		\begin{bettingaxiom}[Risk-neutrality]
			For any $x,y$ in the interval $\D$ of possible net losses, the bookmaker is indifferent between a sure loss of $(x+y)/2$ and an equal chance of losing $x$ or $y$.
		\end{bettingaxiom}
		
		We express this indifference by requiring equal expected loss, $L_P((x+y)/2) = (L_P(x)+L_P(y))/2$, meaning that $L_P$ is affine in the monetary outcome $d$.

		Now consider a bet with a stake of \EUR{1}. The bookmaker collects the stake and pays out $W \geq 0$ euros, so their net loss is $\delta = W - 1$. Since declining gives a sure net loss of zero, the acceptance criterion becomes
		\begin{align*}
			\textnormal{Bookmaker accepts bet } \delta 
				\iff \Ex^P[W - 1] \leq 0
				\iff \Ex^P[W] \leq 1.
		\end{align*}		
		Hence, for a \EUR{1} stake, we can take $L_P(\delta) = 1 + \delta = W$. 
		For an accepted bet, the e-variable is the amount the bookmaker pays out in euros.

	\section{Conditionality principle and law invariance}\label{sec:conditionality}
		In this section, we show that the conditionality of \citet{birnbaum1962foundations} is equivalent to law invariance, when we apply it to the validity status of a decision variable.	
	
		Suppose that a random index $I$ selects independently from a family $(\delta_i)_{i \in \I}$ of decision variables.
		We observe both which index $I = i$ was selected and the resulting decision $\delta_I = d$.
		The conditionality principle of \citet{birnbaum1962foundations} then states that observing $(I, \delta_I) = (i, d)$ has the same `evidential meaning' as observing $\delta_i = d$ had $\delta_i$ been fixed in advance.
		
		Translating conditionality to validity, the selected procedure should have the same validity status whether its index was selected and revealed or fixed in advance.
		To present the axiom, we observe that if the same decision variable $\delta_I$ can also be represented as $\delta_I'$ using the same mixture $I$ over a different family, then the conditionality principle can be viewed as stating that $\delta_i$ and $\delta_i'$ should have the same validity status.
		\begin{birnbaumaxiom}[Conditionality for validity]
			If two independent mixtures with the same finite-valued index satisfy $\delta_I = \delta_I'$ almost surely, then 
			\begin{align*}
				\delta_i\textnormal{ is valid}
					\iff \delta_i' \textnormal{ is valid},
			\end{align*}
			for $I$-almost every $i$.
		\end{birnbaumaxiom}
		
		\begin{proposition}\label{prp:law_invariance_vs_conditionality}
			Validity satisfies conditionality if and only if it satisfies law invariance.
		\end{proposition}
		\begin{proof}
			Suppose that validity is law invariant.
			By independence of the index in each representation and the almost sure equality of the selected variables, for $I$-almost every $i$,
			\begin{align*}
				\L(\delta_i)
    				= \L(\delta_I \mid I=i)
					= \L(\delta'_I \mid I=i)	
    				= \L(\delta'_i).
			\end{align*}
			Law invariance therefore gives the same status to $\delta_i$ and $\delta'_i$.
			
			For the converse, suppose that validity satisfies conditionality.
			We consider two decision variables $\delta_1 \overset{d}{=} \delta_2$ and will show that $\delta_1$ is valid if and only if $\delta_2$ is valid.
			
			Let $I \in \{1, 2\}$ be independent of $(\delta_1, \delta_2)$, with both indices having positive probability.
			The index $I$ is then also independent of $\delta_I$, since
			\begin{align*}
				\L(\delta_I \mid I = 1)
					= \L(\delta_1)
					= \L(\delta_2)
					= \L(\delta_I \mid I = 2).
			\end{align*}
			As a consequence, $I$ is independent of each family $(\delta_1, \delta_2)$ and $(\delta_I, \delta_I)$ separately.
			
			Both these families produce the same decision variable $\delta_I$, when we select a decision variable using $I$.
			Conditionality therefore implies that $\delta_i$ is valid if and only if $\delta_I$ is valid, for $i = 1, 2$.
			Hence, $\delta_1$ and $\delta_2$ have the same validity status, proving law invariance.
		\end{proof}		
		
	\section{Examples for data-dependent levels}\label{sec:testing_examples}
		\begin{example}[Invalidity of a classical test for a data-dependent level]\label{exm:pvalue_family}
			Consider the classical family induced by a valid p-variable $p\sim\textnormal{Unif}[0,1]$,
			\begin{align*}
				\phi(a)
				=
				\begin{cases}
					r_a, & p\leq a,\\
					n, & p>a.
				\end{cases}
			\end{align*}
			
			While this family is valid at every fixed level (since $P(\phi(a)=r_a) = a$), the following example shows that it can be invalid at a data-dependent level.
							
			In particular, suppose that after seeing $p$, we use the level $0.05$ if $p \leq 0.05$ and otherwise we use level $0.10$:
			\begin{align*}
				\widetilde{\alpha}
				=
				\begin{cases}
					0.05, & p\leq0.05,\\
					0.10, & p>0.05.
				\end{cases}
			\end{align*}
			This means $\phi(\widetilde{\alpha})$ claims a rejection at level 5\% ($r_{0.05}$) when $p \leq 0.05$, a rejection at level 10\% ($r_{0.10}$) when $0.05 < p \leq 0.10$, and a non-rejection ($n$) when $p > 0.10$.
			As a consequence,
			\begin{align*}
				P\bigl(\phi(\widetilde{\alpha})=r_{0.05}\bigr)=0.05,\ 
				P\bigl(\phi(\widetilde{\alpha})=r_{0.10}\bigr)=0.05,\
				P\bigl(\phi(\widetilde{\alpha})=n\bigr)=0.90.
			\end{align*}
			
			The expected distortion ratio is
			\begin{align*}
				\frac{0.05}{0.05}+\frac{0.05}{0.10}
				=\frac{3}{2}>1,
			\end{align*}
			so $\phi$ is invalid for $\widetilde{\alpha}$ under criterion \eqref{ineq:characterized_validity}.
			
			We now show how the axioms force the same conclusion by constructing an invalid test with a data-independent level that has the same distribution as $\phi(\widetilde{\alpha})$.
			Let $\widetilde{\alpha}'$ be data-independent, taking the values $0.05$ and $0.10$ with probabilities $2/3$ and $1/3$, respectively, and define the family of tests
			\begin{align*}
				\phi'(a)
				=
				\begin{cases}
					r_a, & p\leq\min\{3a/2,1\},\\
					n, & p>\min\{3a/2,1\}.
				\end{cases}
			\end{align*}
			The Type-I error of $\phi'(a)$ is $\min\{3a/2,1\}$, which exceeds $a$ for every $a\in(0,1)$ so that the family is invalid at every constant level by Axiom \ref{axiom:classical_calibration}.
			By Axiom \ref{axiom:mixing_invalid}, it is then invalid for the data-independent $\widetilde{\alpha}'$.
			
			Moreover, $\phi'(\widetilde{\alpha}')$ has the same distribution as $\phi(\widetilde{\alpha})$:
			\begin{align*}
				P\bigl(\phi'(\widetilde{\alpha}')=r_{0.05}\bigr)
				&=\frac{2}{3}(0.075)=0.05,
				&
				P\bigl(\phi'(\widetilde{\alpha}')=r_{0.10}\bigr)
				&=\frac{1}{3}(0.15)=0.05.
			\end{align*}
			Axiom \ref{axiom:distribution} therefore makes $\phi$ invalid for $\widetilde{\alpha}$ as well.
		\end{example}

		\begin{example}[Conditional validity violates axioms]\label{exm:conditional_type_i_error}
			Recall the conditional notion of validity \eqref{ineq:distortion} that underlies the classical counterargument against testing with a data-dependent level:
			\begin{align*}
				\phi \textnormal{ is valid for } \widetilde{\alpha}
					\iff	 P\left(\phi(\widetilde{\alpha})=r_{\widetilde{\alpha}}\mid\widetilde{\alpha}\right)\leq \widetilde{\alpha}, \quad\textnormal{almost surely}.
			\end{align*}
			
			We show this criterion violates Axiom \ref{axiom:distribution}.
			It does satisfy Axiom \ref{axiom:classical_calibration}, since conditioning on a constant level leaves the rejection probability unchanged.

			Consider again the test family induced by the p-variable from Example \ref{exm:pvalue_family}.
			We now instead consider an overly conservative data-dependent level:
			\begin{align*}
				\widetilde{\alpha}
				=
				\begin{cases}
					0.10, & p\leq 0.05,\\
					0.05, & p>0.05.
				\end{cases}
			\end{align*}
			The test at this data-dependent level conservatively claims a rejection at level 10\% when $p \leq 0.05$:
			\begin{align*}
				\phi(\widetilde{\alpha})
				=
				\begin{cases}
					r_{0.10}, & p\leq0.05,\\
					n, & p>0.05.
				\end{cases}
			\end{align*}
			
			To show a violation of Axiom \ref{axiom:distribution}, consider the conservative test family
			\begin{align*}
				\phi'(a)
				=
				\begin{cases}
					r_a, & p\leq a/2,\\
					n, & p>a/2.
				\end{cases}
			\end{align*}
			Let $\widetilde{\alpha}'\equiv0.10$.
			The resulting tests at their respective levels are identical:
			\begin{align*}
				\phi'(\widetilde{\alpha}') 
					=\phi(\widetilde{\alpha}).
			\end{align*}
			
			Notice that $\phi'(0.10)$ has Type-I error $0.05 \leq 0.10$, so that $\phi'$ is valid for $\widetilde{\alpha}'$ by Axiom \ref{axiom:classical_calibration}.
			Since the tests have the same distribution, Axiom \ref{axiom:distribution} requires $\phi$ to be valid for $\widetilde{\alpha}$ as well.
			
			However, the conditional notion of validity states that $\phi$ is not valid for $\widetilde{\alpha}$, since
			\begin{align*}
				P\left(\phi(\widetilde{\alpha})=r_{\widetilde{\alpha}}\mid\widetilde{\alpha}=0.10\right)
				= 1 > 
				0.10.
			\end{align*}
			As a consequence, this notion of validity violates Axiom \ref{axiom:distribution}.
		\end{example}

	\section{Proofs of the characterization theorems}\label{sec:proof_characterizations}
		Classical calibration already determines validity for ordinary fixed-level tests.
		We first provide a lemma that reproduces an arbitrary decision distribution by an independent random choice among such tests, preserving every decision probability.
		It chooses all these tests to be valid when the original expected loss is at most one, and all to be invalid otherwise.
		Independent selection and law invariance then determine the original decision's validity.

		Here $\mu(B)$ is the probability of rejecting at a level in $B\subseteq(0,1)$; the remaining mass $1-\mu((0,1))$ is the non-rejection probability.
		The new level has distribution $\nu$, and the test at level $a$ rejects with probability $q(a)$.
		The displayed identity says that these choices reproduce $\mu$.
		All level integrals below are over $(0,1)$.

		\begin{lemma}[Representation by binary tests]\label{lem:independent_level_representation}
		    Let $\mu$ be a nonnegative measure on $(0,1)$ with total mass at most one, and put $m=\int a^{-1}\,\mu(da)$, the corresponding expected loss.
		    There exist a probability measure $\nu$ on $(0,1)$ and a measurable function $q:(0,1)\to[0,1]$ such that, for every Borel set $B\subseteq(0,1)$,
		    \begin{align}\label{eq:rejection_factorization}
		        \mu(B)=\int_B q(a)\,\nu(da).
		    \end{align}
		    If $m\le1$, these can be chosen with $q(a)\le a$ for every $a$; if $m>1$, with $q(a)>a$ for every $a$.
		\end{lemma}
		
		\begin{proof}
		    If $m=0$, then $\mu=0$, so take $q\equiv0$ and any probability measure $\nu$.
		    If $0<m\le1$, take $q(a)=ma$, making every test classically valid.
		    Then $\nu(da)=\mu(da)/(ma)$ satisfies \eqref{eq:rejection_factorization} and has total mass $m^{-1}\int a^{-1}\,\mu(da)=1$.
		
		    For $m>1$, the choice $q(a)=ma$ can exceed one.
		    Instead, set $q_t(a)=a+t(1-a)$ for $0<t\le1$, increasing rejection probabilities from $a$ toward one.
		    These tests are all classically invalid because $q_t(a)>a$.
		    Choose $t$ so that the level distribution $\mu(da)/q_t(a)$ has total mass one:
		    \begin{align*}
		        F(t):=\int\frac{\mu(da)}{q_t(a)}=1.
		    \end{align*}
		    The bound $q_t(a)\ge\varepsilon$ on $t\in[\varepsilon,1]$ gives continuity of $F$ on $(0,1]$ by dominated convergence.
		    Monotone convergence gives $F(t)\uparrow m>1$ as $t\downarrow0$, including when $m=\infty$, while $F(1)=\mu((0,1))\le1$.
		    The intermediate value theorem therefore supplies the required $t$.
		    Taking $q=q_t$ and $\nu(da)=\mu(da)/q_t(a)$ proves the claim.
		\end{proof}
		
		\begin{proof}[Proof of Theorem~\ref{thm:characterizing_validity}]
		    Define the rejection measure of $\phi(\widetilde\alpha)$ by
		    \begin{align}\label{eq:rejection_measure}
		        \mu(B)=P\bigl(\phi(\widetilde\alpha)=r_{\widetilde\alpha},\ \widetilde\alpha\in B\bigr).
		    \end{align}
		    Conditioning on $\widetilde\alpha$ identifies the expected loss with the theorem's criterion:
		    \begin{align}\label{eq:distortion_integral}
		        m:=\int\frac{\mu(da)}a
		        &=\Ex^P\!\left[\frac{\ind{\phi(\widetilde\alpha)=r_{\widetilde\alpha}}}{\widetilde\alpha}\right]\notag\\
		        &=\Ex^P\!\left[\frac{P(\phi(\widetilde\alpha)=r_{\widetilde\alpha}\mid\widetilde\alpha)}{\widetilde\alpha}\right].
		    \end{align}
		
		    Apply Lemma~\ref{lem:independent_level_representation} to obtain $\nu$ and $q$.
		    Independently of the data, choose independent $A\sim\nu$ and $U\sim\mathrm{Unif}[0,1]$.
		    Define $\phi'(a)=r_a$ when $U\le q(a)$ and $\phi'(a)=n$ otherwise, so $A$ is data-independent for this test family.
		    Equation~\eqref{eq:rejection_factorization} gives $\phi'(A)\overset d=\phi(\widetilde\alpha)$.
		    The lemma gives $q(a)\le a$ for every $a$ when $m\le1$, and $q(a)>a$ when $m>1$.
		    Since $\phi'(a)$ rejects with probability $q(a)$, classical calibration (Axiom~\ref{axiom:classical_calibration}) makes all these tests valid in the first case and invalid in the second.
		    Because $A$ is data-independent, Axiom~\ref{axiom:mixing_valid} makes $\phi'(A)$ valid in the first case, while Axiom~\ref{axiom:mixing_invalid} makes it invalid in the second.
		    Law invariance (Axiom~\ref{axiom:distribution}) transfers this conclusion to $\phi(\widetilde\alpha)$, which has the same distribution.
		
		    It remains to verify that the criterion itself satisfies the axioms.
		    It is law invariant because $m$ is the expected loss of the decision, and at a constant level $a$ it reduces to $P(\phi(a)=r_a)\le a$.
		    For a data-independent level with distribution $\lambda$, independence and Tonelli's theorem give
		    \begin{align*}
		        m=\int\frac{P(\phi(a)=r_a)}a\,\lambda(da).
		    \end{align*}
		    Averaging ratios at most one, or ratios strictly greater than one, preserves the corresponding inequality, proving both parts of the independent selection axiom.
		\end{proof}
		
		\begin{proof}[Proof of Theorem~\ref{thm:e-value}]
		    Define $\mu(B)=P(\e\in\{r_a:a\in B\})$, so that
		    \begin{align*}
		        m=\int a^{-1}\,\mu(da)=\Ex^P[L(\e)].
		    \end{align*}
		    Lemma~\ref{lem:independent_level_representation} reproduces this decision distribution as a mixture of binary tests, all classically valid if $m\le1$ and all classically invalid if $m>1$.
		    Axioms~\ref{axiom:evalue_calibration} and~\ref{axiom:evalue_mixing} classify this mixture accordingly, and law invariance gives the same conclusion for $\e$.
		
		    The criterion is law invariant and classically calibrated, as in the preceding proof.
		    For arbitrary mixtures, Tonelli's theorem expresses expected loss as the average of the component expected losses, preserving both the bound $\le1$ and the strict inequality $>1$.
		\end{proof}

		\begin{proof}[Proof of Proposition~\ref{prp:invalid_binary_mixture}]
		    Since $\e$ is e-invalid, its expected loss $m=\Ex^P[L(\e)]$ exceeds one.
		    The case $m>1$ of Lemma~\ref{lem:independent_level_representation}, applied to $\mu(B)=P(\e\in\{r_a:a\in B\})$, reproduces its law as an independent mixture of binary tests with rejection probabilities $q(a)>a$, all classically invalid.
		\end{proof}

	\section{Proof for an arbitrary decision space} \label{sec:proof_arbitrary_calibration}

		We use two lemmas, stated and proved after the main proof: the first fixes a common normalization on finite sets, and the second extends the resulting criterion to arbitrary laws.
		Throughout, a binary law may also be a point mass, and we use the convention $0\cdot\infty=0$.

		\begin{proof}[Proof of Theorem~\ref{thm:expected_loss}]
			Assume first that Axioms \ref{axiom:evalue_distribution}, \ref{axiom:evalue_mixing} and \ref{axiom:stability_of_invalidity} hold.
			By law invariance, let $\V$ denote the set of valid laws on $(\D, \A)$.
			Our goal is to show that there exists some measurable $L : \D \to [0, \infty]$ such that for every law $\mu$,
			\begin{align*}
				\mu \in \V \iff \int L\,d\mu \leq 1.
			\end{align*}

			For any candidate $L$, Lemma~\ref{lem:arbitrary_binary_decomposition} states that we can express every law $\mu$ as a mixture of binary laws that agree with it on whether the expected-loss bound holds.
			By Axiom~\ref{axiom:evalue_mixing}, it therefore suffices to establish the representation for binary laws.
			This reduces the proof to showing that there exists a measurable $L:\D\to[0,\infty]$ such that, for every $d,d'\in\D$ and $q\in[0,1]$,
			\begin{align*}
				(1 - q) \mu_d + q \mu_{d'} \in \V
					\iff (1-q) L(d) + q L(d') \leq 1.
			\end{align*}
			
			The strategy in the remainder is to use Lemma~\ref{lem:arbitrary_finite_geometry} to establish the existence of a function that characterizes validity on a finite decision set by an expected-value bound.
			We will show that the functions obtained on different finite sets agree on the decisions they share, so that they define one function on $\D$.
			The lemma requires two constant laws, one valid and one invalid, that have a valid mixture with both weights positive.
			We first suppose such a mixture exists and then treat the case where none exists.

			\noindent\emph{Case 1.}
			To be able to apply Lemma~\ref{lem:arbitrary_finite_geometry}, suppose that there exist $a,b\in\D$ and $r\in(0,1)$ such that $\mu_a\in\V$, $\mu_b\notin\V$ and $(1-r)\mu_a+r\mu_b\in\V$.
			Fix such $a,b,r$ throughout this case.

			On each finite set containing $a,b$, the lemma will give a unique representation with value $0$ at $a$.
			These representations may assign negative values, so we will first use them to define a function $g$ and then obtain a nonnegative $L$.

			The lemma's convexity assumptions follow from independent selection.
			Fix a finite set of decisions $d_1,\ldots,d_n$ containing $a,b$.
			A probability vector $p\in\Delta_n$ specifies the law $\sum_i p_i\mu_{d_i}$.
			Mixing these laws mixes their probability vectors.
			Thus, by Axiom~\ref{axiom:evalue_mixing}, the vectors producing valid laws form a convex set, as do those producing invalid laws.

			The set of vectors producing valid laws is also closed, by Axiom~\ref{axiom:stability_of_invalidity}.
			To see this, suppose $p^{(k)}$ converges to $p$, each law $\sum_i p_i^{(k)}\mu_{d_i}$ is valid, and, for the sake of contradiction, $\mu=\sum_i p_i\mu_{d_i}$ is invalid.
			The axiom supplies $q\in(0,1)$ such that $(1-q)\mu+q\nu$ is invalid for every law $\nu$.
			For all sufficiently large $k$, every coordinate satisfies $p_i^{(k)}\ge(1-q)p_i$, so
			\begin{align*}
				p^{(k)}=(1-q)p+q r^{(k)},\qquad
				r^{(k)}=\frac{p^{(k)}-(1-q)p}{q}\in\Delta_n.
			\end{align*}
			The law corresponding to $p^{(k)}$ would therefore be invalid, which is a contradiction.
			Lemma~\ref{lem:arbitrary_finite_geometry} therefore gives the claimed unique representation on this finite set.

			These representations give one function $g$ satisfying the validity criterion for every binary law.
			For each $d\in\D$, define $g(d)$ to be the value assigned to $d$ by Lemma~\ref{lem:arbitrary_finite_geometry} on $\{a,b,d\}$, with value $0$ at $a$.
			To check a binary law on $\{d,d'\}$, apply the lemma to $\{a,b,d,d'\}$, again with value $0$ at $a$.
			Setting the probabilities of the other decisions to zero restricts this representation to $\{a,b,d\}$ or $\{a,b,d'\}$.
			By uniqueness, the values assigned to $d,d'$ are $g(d),g(d')$, so the binary law is valid exactly when its expected value of $g$ is at most $1$.

			The same argument gives the criterion for every finite law.
			Applying the lemma to any finite decision set together with $a,b$ and restricting to each triple gives
			\begin{align}\label{eq:arbitrary_finite_calibration}
				\sum_{i=1}^n p_i\mu_{d_i}\in\V
				\quad\Longleftrightarrow\quad
				\sum_{i=1}^n p_i g(d_i)\le1.
			\end{align}
			Here $g:\D\to\bR\cup\{\infty\}$, $g(a)=0$, and $g(b)\in(1,\infty)$.

			The function $g$ is measurable because each set $\{d:g(d)\le z\}$ can be described by validity of a particular binary law.
			For fixed $c\in\D$ and $p\in[0,1]$, consider the laws $p\mu_d+(1-p)\mu_c$ as $d$ varies.
			Reporting $d$ with probability $p$ and $c$ otherwise, using a common uniform variable, realizes these laws as a jointly measurable family.
			Footnote~\ref{fn:validity_measurability} therefore makes the set $\{d:p\mu_d+(1-p)\mu_c\in\V\}$ measurable.
			By \eqref{eq:arbitrary_binary_thresholds}, each set $\{d:g(d)\le z\}$, for $z\in\bR$, has this form.

			To obtain a nonnegative function $L$, we first show that $g$ has a finite lower bound.
			Recall that $b$ is the fixed invalid reference decision.
			Its score satisfies $g(b)>1$, while validity of $(1-r)\mu_a+r\mu_b$ and $g(a)=0$ give $r\,g(b)\le1$, so $g(b)$ is finite.
			Axiom~\ref{axiom:stability_of_invalidity}, applied at $\mu_b$, supplies one $p\in(0,1)$, independent of $d$, such that $(1-p)\mu_b+p\mu_d$ is invalid for every $d$.
			By \eqref{eq:arbitrary_finite_calibration},
			\begin{align*}
				(1-p)g(b)+p\,g(d)>1,\qquad
				g(d)>\frac{1-(1-p)g(b)}{p}\quad(d\in\D).
			\end{align*}
			The right-hand side is finite and independent of $d$, giving a uniform lower bound on $g$.
			Consequently $m_*:=\inf_{d\in\D}g(d)$ is finite and, since $g(a)=0$, at most zero.
			Define
			\begin{align*}
				L(d)=\frac{g(d)-m_*}{1-m_*}\in[0,\infty].
			\end{align*}
			Since $\sum_i p_i=1$, we have $\sum_i p_i L(d_i)\le1$ if and only if $\sum_i p_i g(d_i)\le1$.
			Thus $L$ is measurable, nonnegative and gives the required representation for binary laws.

			\noindent\emph{Case 2.}
			Now suppose every mixture of a valid constant law and an invalid constant law is invalid whenever both receive positive weight.
			By independent selection and this assumption, a binary law is valid exactly when every point mass entering it with positive weight is valid.
			Put $C=\{d\in\D:\mu_d\in\V\}$ and define
			\begin{align*}
				L(d)=
				\begin{cases}
					0,&d\in C,\\
					\infty,&d\notin C.
				\end{cases}
			\end{align*}
			The set $C$, and hence $L$, is measurable by footnote~\ref{fn:validity_measurability}, applied to the family of constant variables $\e_d\equiv d$.
			For a binary law $\eta$, the integral $\int L\,d\eta$ is $0$ when $\eta$ is valid and $\infty$ otherwise.
			This gives the required representation in this case as well.

			In both cases, we have obtained a measurable, nonnegative $L$ satisfying the validity criterion for every binary law.
			The initial reduction therefore gives the criterion for every law and completes the forward implication.

			\noindent\emph{Uniqueness.}
			Suppose $0<q_{\delta_0}<1$ for some decision variable $\delta_0$.
			Since every valid decision variable has $q_\delta=1$, the law $\mu_{\delta_0}$ is invalid.
			In Case 2, validity means concentration on $C$, so every mixture giving positive weight to an invalid law is invalid.
			This would give $q_{\delta_0}=0$, a contradiction.
			Hence Case 1 applies, where the constructed loss $L=(g-m_*)/(1-m_*)$ has $\inf L=0$.
			Let $L'$ be any other nonnegative representing loss with $\inf L'=0$.
			The valid mixture $(1-r)\mu_a+r\mu_b$ and invalidity of $\mu_b$ give $L(a)<1<L(b)<\infty$ and $L'(a)<1<L'(b)<\infty$.
			On each triple $\{a,b,d\}$, the scores $(L-L(a))/(1-L(a))$ and $(L'-L'(a))/(1-L'(a))$ represent the same valid laws and are zero at $a$.
			Lemma~\ref{lem:arbitrary_finite_geometry} therefore gives
			\begin{align*}
				\frac{L(d)-L(a)}{1-L(a)}
				=\frac{L'(d)-L'(a)}{1-L'(a)}
				\qquad(d\in\D).
			\end{align*}
			Thus $L'=1+c(L-1)$ with $c=(1-L'(a))/(1-L(a))>0$.
			Taking infima gives $0=1-c$, so $L'=L$.

			\noindent\emph{Converse.}
			Suppose validity is given by $\int L\,d\mu\le1$ for a measurable $L:\D\to[0,\infty]$.
			Law invariance is immediate.

			Both parts of the independent selection axiom follow because the expectation of $L$ under a mixture is the average of its expectations under the component laws.
			For a jointly measurable family $(\e_t)_{t\in T}$ and an independent index $\widetilde t$ with law $\pi$, Tonelli gives
			\begin{align*}
				\Ex[L(\e_{\widetilde t})]
				=\int_T\Ex[L(\e_t)]\,\pi(dt),
			\end{align*}
			with a measurable integrand.
			This also makes $\{t\in T:\e_t\textnormal{ is valid}\}$ measurable, as required in footnote~\ref{fn:validity_measurability}.

			Nonnegativity of $L$ gives Axiom~\ref{axiom:stability_of_invalidity}.
			For an invalid law $\mu$, choose $p\in(0,1)$ with $(1-p)\int L\,d\mu>1$, which is also possible when the integral is infinite.
			For every probability law $\nu$, nonnegativity gives
			\begin{align*}
				\int L\,d\bigl((1-p)\mu+p\nu\bigr)
				\ge(1-p)\int L\,d\mu>1.
			\end{align*}
		\end{proof}

		\begin{lemma}[Finite score representation]\label{lem:arbitrary_finite_geometry}
			Let $\Delta_n$ be the probability simplex with vertices $e_1,\ldots,e_n$.
			Suppose $\V\subseteq\Delta_n$ is closed and both $\V$ and $\Delta_n\setminus\V$ are convex.
			Fix vertices $e_a\in\V$ and $e_b\notin\V$ such that $(1-r)e_a+re_b\in\V$ for some $r\in(0,1)$.
			Then there is a unique vector $g\in(\bR\cup\{\infty\})^n$ with $g_a=0$ such that, for every $p\in\Delta_n$,
			\begin{align*}
				p\in\V
				\quad\Longleftrightarrow\quad
				\sum_{i=1}^n p_i g_i\le1.
			\end{align*}
		\end{lemma}

		\begin{proof}[Proof of Lemma~\ref{lem:arbitrary_finite_geometry}]
			We first identify where validity is possible.
			Let $J$ list the indices receiving positive weight in any valid vector and $F$ be the simplex supported on $J$:
			\begin{align*}
				J=\{i:p_i>0\text{ for some }p\in\V\},\qquad
				F=\{p\in\Delta_n:p_i=0\text{ for }i\notin J\}.
			\end{align*}
			Points outside $F$ are invalid, since $\V\subseteq F$.
			Both valid and invalid points occur in $F$: $e_a\in\V$, while $e_b\in F\setminus\V$ because $(1-r)e_a+re_b\in\V$ has positive $b$-th coordinate, so $b\in J$.

			We seek an affine function whose value determines validity on $F$.
			Its vertex values will supply the scores, since an affine function averages those values.
			Convexity of both sets lets us apply the separating hyperplane theorem: there are a nonconstant affine function $\ell$ on the affine hull of $F$ and a threshold $t$ with $\ell\le t$ on $\V$ and $\ell\ge t$ on $F\setminus\V$.

			We show that points in $F$ with $\ell=t$ are valid by approaching them from $\ell<t$, where validity is guaranteed.
			For this we need a vertex $e_j$ of $F$ with $\ell(e_j)<t$.
			To find one, form a valid average using every index in $J$: choose $p^{(i)}\in\V$ with $p^{(i)}_i>0$ for each $i\in J$.
			Their average $u$ belongs to $\V$ by convexity and has $u_i>0$ for every $i\in J$.
			If all vertex values were at least $t$, then $\ell(u)=\sum_{i\in J}u_i\ell(e_i)\le t$ would force them all to equal $t$, contradicting nonconstancy.
			Thus such an $e_j$ exists, and for $p\in F$ with $\ell(p)=t$, the vectors $(1-s)p+se_j$ have $\ell<t$ and converge to $p$ as $s\downarrow0$.
			Closedness gives
			\begin{align*}
				\V=\{p\in F:\ell(p)\le t\}.
			\end{align*}

			We now normalize the score of $e_a$ to $0$ and the threshold to $1$.
			We have $\ell(e_a)<t$: invalidity of $e_b$ gives $\ell(e_b)>t$, and validity of $(1-r)e_a+re_b$ gives $(1-r)\ell(e_a)+r\ell(e_b)\le t$.
			Define
			\begin{align*}
				g_i=
				\begin{cases}
					\dfrac{\ell(e_i)-\ell(e_a)}{t-\ell(e_a)},&i\in J,\\[4pt]
					\infty,&i\notin J.
				\end{cases}
			\end{align*}
			Then $g_a=0$, and for $p\in F$ affinity gives
			\begin{align*}
				\sum_i p_i g_i=\frac{\ell(p)-\ell(e_a)}{t-\ell(e_a)}\le1
				\quad\Longleftrightarrow\quad \ell(p)\le t.
			\end{align*}
			If $p\notin F$, some $p_i>0$ has $i\notin J$; assigning $g_i=\infty$ makes $\sum_i p_i g_i=\infty>1$.

			For uniqueness, first recover $g_b$ from $\V$.
			Any representing score with $g_a=0$ satisfies $1<g_b\le1/r<\infty$, since $e_b$ is invalid and $(1-r)e_a+re_b$ is valid.
			The vector $(1-q)e_a+qe_b$ is valid exactly when $qg_b\le1$, so
			\begin{align*}
				g_b=\frac{1}{\max\{q\in[0,1]:(1-q)e_a+qe_b\in\V\}}.
			\end{align*}
			The maximum exists by closedness.

			We recover every remaining $g_i$ by testing $g_i\le z$ for each real $z$.
			Choose mixture weights so that the score would equal $1$ if $g_i=z$.
			Mix $e_i$ with $e_a$ for $z\ge1$; for $z<1$, mix with $e_b$, taking $p_z=(g_b-1)/(g_b-z)$ so that $(1-p_z)g_b+p_z z=1$.
			Thus
			\begin{align}\label{eq:arbitrary_binary_thresholds}
				g_i\le z
				\quad\Longleftrightarrow\quad
				\begin{cases}
					(1-z^{-1})e_a+z^{-1}e_i\in\V,&z\ge1,\\[2pt]
					(1-p_z)e_b+p_z e_i\in\V,&z<1,
				\end{cases}.
			\end{align}
			Thus $\V$ uniquely determines every coefficient, including infinite ones.
		\end{proof}

		\begin{lemma}[Binary decomposition]\label{lem:arbitrary_binary_decomposition}
			Let $(\D,\A)$ be a measurable space, $L:\D\to[0,\infty]$ measurable, and $\mu$ a probability law.
			Then $\mu$ is a mixture of binary laws $\eta$ satisfying
			\begin{align*}
				\int L\,d\eta\leq1
				\quad\Longleftrightarrow\quad
				\int L\,d\mu\leq1.
			\end{align*}
			The mixture can be realized by a jointly measurable family and an independent random index, without additional assumptions on $(\D,\A)$.
		\end{lemma}

		\begin{proof}[Proof of Lemma~\ref{lem:arbitrary_binary_decomposition}]
			First let $m=\int L\,d\mu<\infty$.
			Put
			\begin{align*}
				D_-=\{L<m\},\qquad D_+=\{m<L<\infty\},\qquad D_0=\{L=m\},
			\end{align*}
			and let $\mu_{D_-},\mu_{D_+},\mu_{D_0}$ denote the restrictions of $\mu$.
			Since $\mu(L=\infty)=0$, the expected deficit and excess agree:
			\begin{align*}
				K=\int_{D_-}(m-L)\,d\mu=\int_{D_+}(L-m)\,d\mu<\infty.
			\end{align*}
			We will decompose $\mu$ into binary laws with expected loss $m$, so that the expected-loss bound of $1$ holds for each component exactly when it holds for $\mu$.

			Suppose $K=0$.
			Then $\mu(D_0)=1$.
			This means we can decompose $\mu$ into a mixture of the point masses $\mu_z$, $z\in D_0$, with mixing measure $\mu$ itself.
			Indeed, for every $B\in\A$, the probability assigned by this mixture is
			\begin{align*}
				\int_{D_0}\mu_z(B)\,\mu(dz)=\mu(B\cap D_0)=\mu(B),
			\end{align*}
			so the mixture equals $\mu$.
			Each point mass is a binary law with the required expected loss $L(z)=m$.

			Now suppose $K>0$.
			For every pair $(x,y)\in D_-\times D_+$, the losses $L(x)$ and $L(y)$ lie on opposite sides of $m$, so we can construct a binary law $\eta_{x,y}$ on these two points with expected loss $m$.
			Let $q(x,y)$ be its probability of returning $x$, with the remaining probability assigned to $y$.
			Solving $q(x,y)L(x)+(1-q(x,y))L(y)=m$ gives
			\begin{align*}
				q(x,y)&=\frac{L(y)-m}{L(y)-L(x)},& \eta_{x,y}&=q(x,y)\mu_x+(1-q(x,y))\mu_y.
			\end{align*}

			We now choose a measure $\Lambda$ on $D_-\times D_+$ to weight these laws so that they reproduce $\mu_{D_-}+\mu_{D_+}$.
			Weighting $\mu_{D_-}(dx)\mu_{D_+}(dy)$ by $L(y)-L(x)$ cancels the denominators in $q$ and $1-q$.
			Integrating the remaining numerators over $y$ and $x$, respectively, gives $K$, so we divide by $K$ and define
			\begin{align*}
				\Lambda(dx,dy)=\frac{L(y)-L(x)}K\,\mu_{D_-}(dx)\mu_{D_+}(dy).
			\end{align*}
			Indeed, for every $B\in\A$, the $x$ and $y$ outcomes contribute exactly $\mu_{D_-}(B)$ and $\mu_{D_+}(B)$:
			\begin{align*}
				\int_{D_-\times D_+}q(x,y)\mu_x(B)\,d\Lambda &=\frac{\int_{D_+}(L-m)\,d\mu}{K}\,\mu_{D_-}(B)=\mu_{D_-}(B),\\
				\int_{D_-\times D_+}(1-q(x,y))\mu_y(B)\,d\Lambda &=\frac{\int_{D_-}(m-L)\,d\mu}{K}\,\mu_{D_+}(B)=\mu_{D_+}(B).
			\end{align*}
			This recovers $\mu$ on $D_-\cup D_+$, leaving its restriction to $D_0$.
			Taking $B=\D$ and adding the two identities gives $\Lambda(D_-\times D_+)=1-\mu(D_0)$.
			Adding the point masses on $D_0$ with mixing measure $\mu_{D_0}$ supplies the remaining weight $\mu(D_0)$ and yields the probability mixture
			\begin{align*}
				\mu=\int_{D_-\times D_+}\eta_{x,y}\,\Lambda(dx,dy) +\int_{D_0}\mu_z\,\mu(dz).
			\end{align*}
			In both cases, every component has expected loss $m$, proving the required equivalence.

			To realize these mixtures measurably, let a random index $I$ identify the component law, and let $W\sim\operatorname{Unif}[0,1]$, independent of $I$, determine the outcome within that component.
			The law of $I$ is the mixing distribution just constructed: $\mu_{D_0}$ on $D_0$ if $K=0$; if $K>0$, it assigns measure $\Lambda$ to $D_-\times D_+$ and $\mu_{D_0}$ to $D_0$, regarded as disjoint sets of indices.
			For $w\in[0,1]$, define $X_z(w)=z$ at a point-mass index $z\in D_0$; at a pair index $(x,y)\in D_-\times D_+$ when $K>0$, define $X_{(x,y)}(w)=x$ if $w\le q(x,y)$, and $y$ otherwise.
			These formulas are jointly measurable in the index and $w$, since the coordinate maps and $q$ are measurable.
			For each fixed index $i$, $X_i(W)$ has the prescribed component law.
			Since $I$ has the mixing distribution and is independent of $W$, $X_I(W)$ has law $\mu$.

			Finally, when $\int L\,d\mu=\infty$, it suffices that every binary component have expected loss greater than $1$.
			To use the finite-mean construction, truncate $L$ at a finite $M$ with $m_M:=\int(L\wedge M)\,d\mu>1$; such an $M$ exists by monotone convergence.
			Applying that construction to $L\wedge M$ decomposes the same law $\mu$ into binary laws with expected truncated loss $m_M$.
			Since $L\ge L\wedge M$, their expected losses are at least $m_M>1$, as required.
		\end{proof}

	\section{Distinguishing non-rejections}\label{sec:proof_distinct_non_rejections}
		
			\begin{proof}[Proof of Proposition~\ref{prp:distinct_non_rejection_losses}]
				By calibration, $\mu_{r_\alpha}$ is invalid and $(1-\alpha)\mu_{n_\alpha}+\alpha\mu_{r_\alpha}$ is valid.
				Thus $q_{r_\alpha}\geq\alpha>0$, while stability gives $q_{r_\alpha}<1$.
				Theorem~\ref{thm:expected_loss} now gives a unique normalized loss $L$, with $1<L(r_\alpha)<\infty$ by the same two laws.
				The latter law must have expected loss exactly $1$: otherwise a slightly larger rejection probability would still be valid.
				This gives \eqref{eq:distinct_non_rejection_calibration}, including $L(n_\alpha)<1$.
				Since all rejection losses exceed $1$, normalization gives $\inf_\alpha L(n_\alpha)=0$.

				Conversely, Theorem~\ref{thm:expected_loss} gives Axioms \ref{axiom:evalue_distribution}, \ref{axiom:evalue_mixing} and \ref{axiom:stability_of_invalidity}.
				A matched-pair law with rejection probability $p$ has expected loss
				\begin{align*}
					(1-p)L(n_\alpha)+pL(r_\alpha)
						= L(n_\alpha)+\frac{p}{\alpha}\bigl(1-L(n_\alpha)\bigr).
				\end{align*}
				Since $L(n_\alpha)<1$, this is at most $1$ if and only if $p\leq\alpha$, giving classical calibration.
			\end{proof}

			\begin{proof}[Proof of Proposition~\ref{prp:power_non_rejections}]
				Let $\e$ be any maximizing evidence variable, and write $A=\{\e=r_\alpha\}$, $p=P(A)$ and $q=Q(A)<1$.
				By Proposition~\ref{prp:distinct_non_rejection_losses}, validity means $\Ex^P[L(\e)]\leq 1$, and $1<L(r_\alpha)<\infty$.
				Hence $p\leq 1/L(r_\alpha)<1$.
				Write
				\begin{align*}
					C=\Ex^P[L(\e)\mathbf{1}_{A^c}].
				\end{align*}
				Suppose that $C>0$.
				Since $\inf_\beta L(n_\beta)=0$, we can choose $\beta$ such that $(1-p)L(n_\beta)<C$.
				Replacing every output on $A^c$ by $n_\beta$ preserves power and gives expected loss
				\begin{align*}
					pL(r_\alpha)+(1-p)L(n_\beta)
						< pL(r_\alpha)+C
						=\Ex^P[L(\e)]\leq 1.
				\end{align*}
				Now introduce an auxiliary Bernoulli variable with parameter $t\in(0,1)$, independent of $\e$ under both $P$ and $Q$ and with the same distribution under both measures.
				On $A$, continue to report $r_\alpha$; on $A^c$, report $r_\alpha$ if the Bernoulli variable equals one, and $n_\beta$ otherwise.
				For sufficiently small $t$, the resulting expected loss
				\begin{align*}
					pL(r_\alpha)+(1-p)L(n_\beta)
						+t(1-p)\bigl(L(r_\alpha)-L(n_\beta)\bigr)
				\end{align*}
				remains at most one, while power increases to $q+t(1-q)>q$.
				This contradicts optimality, so $C=0$.

				Hence $L(\e)=0$ $P$-almost surely on $A^c$.
				Since $P(A^c)>0$ and every rejection has positive loss, a zero-loss non-rejection exists.
				The pairwise distinct non-rejection losses make this decision unique; denote it by $n$.
				It follows that $\e\in\{n,r_\alpha\}$ $P$-almost surely.
				If $Q(\e\notin\{n,r_\alpha\})>0$, replacing these outputs by $r_\alpha$ would increase power without changing expected loss under $P$, again contradicting optimality.
				Thus $\e\in\{n,r_\alpha\}$ also $Q$-almost surely.
			\end{proof}

	\section{Proof of Theorem~\ref{thm:numerical_maximality}}\label{sec:proof_without_calibration}
		\begin{proof}[Proof of Theorem~\ref{thm:numerical_maximality}]
			Let $\e$ be a valid evidence variable, and let $Q$ be the quantile function of its law with $Q(0) = 0$.
			 On the auxiliary probability space $[0,1)$ equipped with uniform probability, define
			\begin{align*}
				\e_t(u)
					:= Q((u+t) \bmod 1), \quad t, u \in [0, 1).
			\end{align*}
			For every fixed $t$, shifting a uniform random variable on $[0, 1)$ by $t$ modulo $1$ leaves it uniformly distributed.
			Hence, every $\e_t$ has the same law as $\e$ and is therefore valid by Axiom \ref{naxiom:law_invariance}.

			For every fixed outcome $u$,
			\begin{align*}
				\int_0^1 \e_t(u)\,dt
					= \int_0^1 Q(v)\,dv
					= \Ex^P[\e].
			\end{align*}
			Axiom \ref{naxiom:valid_averaging} makes this constant, possibly infinite, valid.
			Normalization therefore forces $\Ex^P[\e]\leq1$, so every validity class satisfying the axioms is contained in $\{\e:\Ex^P[\e]\leq1\}$.
			Conversely, this class satisfies law invariance and normalization.
			It also satisfies the averaging axiom, since Tonelli's theorem gives
			\begin{align*}
				\Ex^P\left[\int \e_t\,\mu(dt)\right]
				= \int \Ex^P[\e_t]\,\mu(dt)
				\leq 1.
			\end{align*}
			As a consequence, it is the largest class satisfying the axioms.
		\end{proof}

\end{document}